\pdfoutput=1
\documentclass[microtype]{gtpart}
\usepackage{graphicx}
\usepackage[mathscr]{eucal}
\usepackage{amssymb}
\usepackage{xcolor}
\definecolor{darkred}{RGB}{200, 60, 0}
\definecolor{mildblue}{RGB}{0, 100, 250}
\usepackage{enumerate, cite}
\usepackage[margin=1.25in]{geometry}
\usepackage{hyperref}
\hypersetup{
    colorlinks=true,
    linkcolor=darkred,
    urlcolor=darkred,
    citecolor=mildblue,      
    urlcolor=darkred,
}
\usepackage[nameinlink]{cleveref}

\newcommand{\br}{\mathbb{R}}

\newcommand{\bz}{\mathbb Z}
\newcommand{\bn}{\mathbb N}

\newcommand{\vp}{\varphi}

\newcommand{\ssm}{\smallsetminus}

\DeclareMathOperator{\mcg}{MCG}

\DeclareMathOperator{\Homeo}{Homeo}

\DeclareMathOperator{\supp}{supp}
\DeclareMathOperator{\Sym}{Sym}
\DeclareMathOperator{\PH}{PH}
\DeclareMathOperator{\HH}H
\DeclareMathOperator{\F}F

\renewcommand{\co}{\colon\thinspace}

\newtheorem{Thm}{Theorem}[section]
\newtheorem{Thm*}{Theorem}
\newtheorem{Prop}[Thm]{Proposition}
\newtheorem{Lem}[Thm]{Lemma}
\newtheorem{Cor}[Thm]{Corollary}
\newtheorem{Cor*}[Thm*]{Corollary}

\newtheorem*{MainThm1}{\Cref{thm:normal generators 1}}
\newtheorem*{MainThm2}{\Cref{thm:mainthm2}}
\newtheorem*{MainThm3}{\Cref{thm:main3}}
\newtheorem*{MainThm4}{\Cref{thm:ses}}
\newtheorem*{MainThm5}{\Cref{thm:normal generators F}}
\newtheorem*{MainThm5'}{\Cref{thm:perfect F}}
\newtheorem*{MainThm6}{\Cref{thm:abelianization}}
\newtheorem*{MainThm7}{\Cref{thm:cardinality}}

\theoremstyle{definition}
\newtheorem{Def}[Thm]{Definition}

\newtheorem*{Ex*}{Examples}
\newtheorem{Rem}[Thm]{Remark}

\numberwithin{equation}{section}

\title{Algebraic and geometric properties of \\homeomorphism groups of ordinals}

\author{Megha Bhat}
\address{Department of Mathematics \\ CUNY Graduate Center \\ New York, NY 10016}
\email{mbhat@gradcenter.cuny.edu}

\author{Rongdao Chen}
\address{Department of Mathematics \\ CUNY Graduate Center \\ New York, NY 10016}
\email{rchen1@gradcenter.cuny.edu}

\author{Adityo Mamun}
\address{Department of Mathematics \\ CUNY Queens College \\ Flushing, NY 11367}
\email{adityo.mamun23@qmail.cuny.edu}

\author{Ariana Verbanac}
\address{Department of Mathematics \\ Temple University \\ Philadelphia, PA 19122}
\email{ariana.verbanac@temple.edu}

\author{Eric Vergo}
\address{Independent researcher}
\email{ericvergo@gmail.com}

\author{Nicholas G. Vlamis}
\address{Department of Mathematics \\ CUNY Graduate Center \\ New York, NY 10016, and \newline Department of Mathematics \\ CUNY Queens College \\ Flushing, NY 11367}
\email{nvlamis@gc.cuny.edu}

\begin{document}  

\begin{abstract}
We study the homeomorphism groups of ordinals equipped with their order topology, focusing on successor ordinals whose limit capacity is also a successor.  
This is a rich family of groups that has connections to both permutation groups and homeomorphism groups of manifolds. 

For ordinals of Cantor--Bendixson degree one, we prove that the homeomorphism group is strongly
distorted and uniformly perfect, and we classify its normal generators. 
As a corollary, we recover and provide a new proof of the classical result that the subgroup of finite permutations in the symmetric group on a countably infinite set is the maximal proper normal subgroup.

For ordinals of higher Cantor--Bendixson degree, we establish a semi-direct product decomposition of the (pure) homeomorphism group.  
When the limit capacity is one, we further compute the abelianizations and determine normal generating sets of minimal cardinality for
these groups.

\end{abstract}

\maketitle

\vspace{-0.5in}


\section{Introduction}
Recall that an \emph{ordinal} is the order-isomorphism class of a totally ordered set that is \emph{well-ordered} in the sense that every nonempty subset has a smallest element. 
Every totally ordered set admits a topology---\emph{the order topology}---with a subbasis given by the open rays \( \{ x: x < y \} \) and \( \{x: y < x \} \) ranging over the elements \( y \) in the set.
As an order isomorphism between totally ordered sets induces a homeomorphism between their respective order topologies, an ordinal has a well-defined homeomorphism type, allowing us to study the homeomorphism group of an ordinal. 

Motivated by recent results regarding homeomorphism groups of surfaces, we investigate several algebraic and geometric properties of homeomorphism groups of ordinals.
We provide a discussion of this motivation from geometric topology after introducing our main results.

For an introduction to ordinals and the terminology that follows, see \Cref{sec:ordinals}. 
Our results focus on successor ordinals, or equivalently, the compact ordinals. 
The homeomorphism type of an infinite successor ordinal is classified by two parameters: an ordinal \( \alpha \) (called the \emph{limit capacity}) and a natural number \( d \) (called the \emph{coefficient}). 
The limit capacity and the coefficient are invariants of a well-ordered set that can be computed topologically as follows: the limit capacity is equal to the predecessor of the \emph{Cantor--Bendixson rank} and the coefficient is equal to the \emph{Cantor--Bendixson degree}.
We study the successor ordinals whose limit capacity is itself a successor ordinal. 

Let us fix some notation.
Given an ordinal \( \alpha \) and a natural number \( d \), we let \( X_{\alpha,d} \) denote a well-ordered set representing a compact ordinal whose limit capacity is \( \alpha + 1 \) and whose coefficient is \( d \). 
We now let \( \HH_{\alpha,d} \) denote \( \Homeo(X_{\alpha,d}) \), the group consisting of homeomorphisms \( X_{\alpha,d} \to X_{\alpha,d} \). 
In \Cref{sec:ordinals}, we will introduce a canonical choice for \( X_{\alpha,d} \) using von Neumann ordinals and the classification of ordinals up to homeomorphism; in particular, in later sections, we set \( X_{\alpha,d} = \omega^{\alpha+1}\cdot d + 1 \).

Our first slate of results is for the case when the coefficient is one. 
Groups in this class generalize \( \Sym(\bn) \), the symmetric group on the natural numbers, as \( \Sym(\bn) \) is isomorphic to \(\HH_{0,1} \).

An element \( g \) in a group \( G \) is a \emph{normal generator} of \( G \) if every element of \( G \) can be expressed as a product of conjugates of \( g \) and \( g^{-1} \); it is a \emph{uniform normal generator} of \( G \) if there exists \( k \in \bn \)\footnote{Throughout, we use the convention that \( 0 \notin \bn \).} such that every element of \( G \) can be expressed as a product of at most \( k \) conjugates of \( g \) and \( g^{-1} \), and the minimum such \( k \) is called the \( g \)-width of \( G \). 

Given \( x \in X_{\alpha,d} \), the set \( \{ y : y \leq x \} \) is a well-ordered set, allowing us to define the \emph{capacity} of \( x \) to be equal to the limit capacity of \( \{ y: y \leq x \} \).

\begin{MainThm1}[Normal generators]
    Let \( \alpha \) be an ordinal. 
		For \( h \in\HH_{\alpha,1} \), the following are equivalent:
    \begin{enumerate}[(i)]
				\item \( h \) normally generates   \(\HH_{\alpha,1} \).
				\item \( h \) uniformly normally generates \(\HH_{\alpha,1} \).
				\item \( h \) induces an infinite permutation of the capacity \( \alpha \) elements of \( X_{\alpha,1} \).  
    \end{enumerate}
		Moreover, if one---and hence all---of the above conditions are satisfied, then the \( h \)-width of \(\HH_{\alpha,d} \) is at most twelve. 
\end{MainThm1}

The proof of Theorem~\ref{thm:normal generators 1} relies on a uniform fragmentation result---a generalization of a lemma of Galvin \cite{GalvinGenerating} from permutation groups---and a technique from Anderson's proof \cite{AndersonAlgebraic} of the algebraic simplicity of several transformation groups, including the homeomorphism group of the Cantor set and the homeomorphism group of the 2-sphere.
As an immediate consequence, \(\HH_{\alpha,1} \) has a maximal proper normal subgroup. 

\begin{Cor}
\label{cor:max subgroup}
Let \( \alpha \) be an ordinal. 
The subgroup  of \(\HH_{\alpha,1} \) consisting of homeomorphisms that induce finite permutations on the capacity \( \alpha \) elements in \( X_{\alpha,1} \) is the maximal proper normal subgroup of \( H_{\alpha,1} \), that is, it contains every proper normal subgroup.  
\qed
\end{Cor}

As \(\HH_{0,1} \) is isomorphic to \( \Sym(\bn) \), \Cref{cor:max subgroup} recovers the classical theorem of Schreier--Ulam \cite{Schreier1933} that the subgroup of finite permutations in \( \Sym(\bn) \) is its maximal normal subgroup. 
The proof presented here is arguably simpler than the standard proof given in classical texts such as \cite{ScottGroup}, which relies on cycle decompositions. 
Similarly, \Cref{thm:normal generators 1} can be viewed as a generalization of a theorem of Bertram \cite{BertramTheorem} showing that all infinite permutations in \( \Sym(\bn) \) are uniform normal generators.
Unlike Bertram's proof, our proof does not rely on cycle decompositions. 
Bertram proves that any permutation in \( \Sym(\bn) \) can be expressed as a product of four conjugates of any infinite permutation, raising the question of what the optimal constant  in \Cref{thm:normal generators 1} is. 

Recall that the \emph{commutator subgroup} of a group \( G \), denoted \( [G,G] \), is  the subgroup generated by the set \( \{ [x,y] : x,y \in G \} \), where \( [x,y] = xyx^{-1}y^{-1} \).
A group \( G \) is \emph{perfect} if it is equal to its commutator subgroup.
It is \emph{uniformly perfect} if there exists \( k \in \bn \) such that every element of \( G \) can be expressed as a product of \( k \) commutators; the minimum such \( k \) is called the \emph{commutator width} of \( G \).  

It is not difficult to find a commutator in \(\HH_{\alpha,1} \) that induces an infinite permutation of the capacity \( \alpha \) elements in \( X_{\alpha,1} \); \Cref{thm:normal generators 1} then implies that \( H_{\alpha,1} \) is a uniformly perfect group of commutator width at most twelve. 
A direct proof, which is simpler than the proof of \Cref{thm:normal generators 1}, can be given of this fact that provides a better bound on the commutator width of \( G \).
This is our second main result.

\begin{MainThm2}
	If \( \alpha \) is an ordinal, then \(\HH_{\alpha,1} \) is uniformly perfect, and the commutator width is at most three. 
\end{MainThm2}

A theorem of Ore \cite{OreSome} says that every element of \( \Sym(\bn) \) can be expressed as a single commutator. 
It is natural to ask whether the same is true for \(\HH_{\alpha,1} \). 

Continuing in analogy with \( \Sym(\bn) \), we turn to a geometric property.
In \cite{BergmanGenerating}, Bergman showed that any left-invariant metric on \( \Sym(\bn) \) has bounded diameter; in the literature, this property is referred to as \emph{strong boundedness} or the \emph{Bergman property}. 
There are several equivalent definitions of strongly bounded, one of relevance here is that every group action on a metric space has bounded orbits.
As a consequence, any homomorphism from a strongly bounded group to a countable group has finite image, a fact we use below.
Note that strong boundedness is a property of a group as a \emph{discrete} group, so there are no continuity requirements for the metrics, actions, or homomorphisms in the  above discussion. 

Adapting a technique appearing several places in the literature (see \cite[Construction~2.3]{LeRouxStrong}), we establish \( \HH_{\alpha,1} \) is strongly distorted.
The notion of strong distortion was introduced by Calegari--Freedman in \cite{CalegariDistortion}, where they established strong distortion for  homeomorphism groups of spheres; in the appendix of the same paper,  Cornulier showed that if \( G \) is strongly distorted, then \( G \) is strongly bounded. 
See \Cref{sec:strong distortion} for the definition of strong distortion.

\begin{MainThm3}
	For every ordinal \( \alpha \), the group \( \HH_{\alpha,1} \) is strongly distorted.
\end{MainThm3}

Strong distortion also implies several other properties, including uncountable cofinality and the Schreier property; we refer the reader to the introduction of \cite{LeRouxStrong} for a thorough discussion.
As strong distortion implies strong boundedness, Theorem~\ref{thm:main3}, in the case \( \alpha = 1 \), establishes that \( \mathrm{Sym}(\bn) \) is strongly bounded, recovering Bergman's  result \cite{BergmanGenerating} and providing a simpler proof. 

We finish our discussion of the degree one case by observing that we cannot drop the hypothesis that the limit capacity is a successor. 
It is a consequence of \cite[Theorem~3.9]{DowThin} that every countable group is a quotient of \( \Homeo(\omega_1) \), where \( \omega_1 \) is the first uncountable ordinal. 
As a consequence, \( \Homeo(\omega_1) \) can be neither uniformly perfect, strongly bounded, nor strongly distorted.

We now turn to the case of coefficient, or equivalently Cantor--Bendixson degree, greater than one, investigating the algebraic structure of \( \HH_{\alpha,d} \) and establishing the failure of the above results when \( d > 1 \).
There are several normal subgroups of \( \HH_{\alpha,d} \) that are essential to our results, which we must introduce.

The first is \( \F_{\alpha,d} \), the subgroup consisting of the homeomorphisms of \( X_{\alpha,d} \) inducing a finite permutation of the capacity \( \alpha \) elements of \( X_{\alpha,d} \).
The second is obtained by equipping \( \HH_{\alpha,d} \) with the compact-open topology and taking the closure of \( \F_{\alpha,d} \), which we denote \( \overline \F_{\alpha,d} \). 
Lastly, we let \( \PH_{\alpha,d} \) denote the subgroup of \( \HH_{\alpha,d} \) fixing each element \( X_{\alpha,d} \) of capacity \( \alpha + 1 \).
Note that, as there are only \( d \) capacity \( \alpha +1 \) elements, \( \PH_{\alpha,d} \) is a finite-index subgroup, and the corresponding quotient group is isomorphic to \( \mathrm{Sym}(d) \), the symmetric group on \( d \) letters. 
Observe that \( \HH_{\alpha,1} = \PH_{\alpha, 1} = \overline \F_{\alpha,1} \), where the last equality follows from \Cref{thm:normal generators 1}.

Recall that a short exact sequence \( 1 \to A \to B \to C \to 1 \) is \emph{split} if the map \( B \to C \) admits a section. 
Also note that, equipped with the compact-open topology, \( \HH_{\alpha,d} \) is a topological group (see \Cref{sec:ordinals}).

\begin{MainThm4}

	Let \( \alpha \) be an ordinal, and let \( d \in \bn \ssm\{1\} \). 
	
	\begin{enumerate}[(1)]
	
		\item There exists a homomorphism \( \chi \co \PH_{\alpha,d} \to \bz^{d-1} \) such that 
				\[
					1 \longrightarrow \overline \F_{\alpha,d} \lhook\joinrel\longrightarrow \PH_{\alpha,d} \overset{\chi}\longrightarrow \bz^{d-1} \longrightarrow 1
				\]
				is a split short exact sequence. 
				
			\item The short exact sequence 
				\[
					1 \longrightarrow  \PH_{\alpha,d} \lhook\joinrel\longrightarrow \HH_{\alpha,d} \longrightarrow \Sym(d) \longrightarrow 1
				\]
				is split.
				
	\end{enumerate}
	
	Moreover, all the homomorphisms in the above exact sequences are continuous when \( \HH_{\alpha,d} \), \( \overline \F_{\alpha,d} \), and \( \PH_{\alpha,d} \) are equipped with the compact-open topology and \( \mathbb Z^{d-1} \) and \( \Sym(d) \) are equipped with the discrete topology.
\end{MainThm4}

If a topological group \( G \) is abstractly the semidirect product of its subgroups \( N \) and \( K \), then this semidirect product is \emph{topological} if both \( N \) and \( H \) are closed and the map \( N \times H \to G \) given by \( (n,h) \mapsto nh \) is a homeomorphism. 

As \( \chi \) is continuous with discrete image, the restriction of \( \chi \) to the image of a section is a homeomorphism, implying the image of any section is discrete. 
Moreover, the continuity of \( \chi \) and of the homomorphism \( \PH_{\alpha,d} \to \Sym(d) \) implies that \( \overline \F_{\alpha,d} \) and \( \PH_{\alpha,d} \) are clopen subgroups, respectively, from which it readily follows that \( \PH_{\alpha,d} \) is homeomorphic to \( \F_{\alpha,d} \times \mathbb Z^{d-1} \) and \( \HH_{\alpha,d} \) is homeomorphic to \( \PH_{\alpha,d} \times \Sym(d) \).

Recall that a short exact sequence is  split if and only if the middle group is a semidirect product of the other two groups.
Therefore, \Cref{thm:ses} yields a topological semi-direct product structure for \( \PH_{\alpha,d} \) and \( \HH_{\alpha,d} \), which we record in the following lemma.

\begin{Cor}(Topological semidirect product decomposition)
	Let \( \alpha \) be an ordinal, and let \( d \in \bn \) with \( d > 1 \).
	Both \( \PH_{\alpha,d} \) and \( \HH_{\alpha,d} \), equipped with the compact-open topology, admit topological semidirect product decompositions.
	Specifically, 
		\[ \PH_{\alpha,d} = \overline \F_{\alpha,d} \rtimes \bz^{d-1} \] 
	and 
	\[ \HH_{\alpha,d} = PH_{\alpha,d} \rtimes \Sym(d) = \left(\overline \F_{\alpha,d} \rtimes \bz^{d-1}\right) \rtimes \Sym(d),\] 
	where \( \mathbb Z^{d-1} \) and \( \Sym(d) \) are discrete and are identified with their images under the sections from \Cref{thm:ses}.
\end{Cor}

An explicit description of \( \chi \), and its section, is given in \Cref{sec:ses}, which comes from \( d-1 \) independent \emph{flux homomorphisms} counting the shifting of the capacity \( \alpha \) elements of \( X_{\alpha,d} \). 
As a consequence, we see that  \( \PH_{\alpha,d} \) and \( \HH_{\alpha,d} \) are neither perfect nor strongly bounded. 
For the latter, note that \( \PH_{\alpha,d} \) has a countably infinite quotient and hence is not strongly bounded per the discussion above, and as \( \HH_{\alpha,d} \) has a finite-index subgroup that fails to be strongly bounded, it must also fail to be strongly bounded. 

A topological group is \emph{coarsely bounded} if every continuous left-invariant metric on the group has bounded diameter (see \cite{RosendalBook}). 
Note that this is a weaker condition than being strongly bounded, as we have required the metrics under consideration to be continuous; in particular, a group is strongly bounded if and only if it is coarsely bounded as a discrete group. 
As mentioned in \Cref{thm:ses}, the homomorphism \( \chi \) is continuous, and hence it follows from \Cref{thm:ses} that \( \PH_{\alpha,d} \) (and hence \( \HH_{\alpha,d} \)) is not coarsely bounded.
Together with \Cref{thm:main3}, this yields the following corollary. 

\begin{Cor}
\label{cor:d=1}
	Let \( \alpha \) be an ordinal, and let \( d \in \bn \).
	The following are equivalent:
	\begin{enumerate}[(i)]
	
		\item \( \HH_{\alpha,d} \) is strongly distorted.
		\item \( \HH_{\alpha,d} \) is strongly bounded.
		\item \( \HH_{\alpha,d} \) is coarsely bounded.
		\item \( d = 1 \). 
		
	\end{enumerate}
\end{Cor}

Recently, Branman--Domat--Hoganson--Lyman \cite[Theorem~A]{BranmanGraphical} showed that the homeomorphism group of a countable successor ordinal is coarsely bounded if and only if its coefficient is one. 
\Cref{cor:d=1} is a strengthening of this statement under the additional hypothesis that the limit capacity is a successor ordinal, and it also broadens the statement beyond the countable ordinals.
In the case of a countable coefficient one successor ordinal whose limit capacity is a limit ordinal, we do not know if its homeomorphism group is strongly bounded, despite knowing that it is coarsely bounded. 

Before moving on, we stress that \Cref{cor:d=1} considers \( \HH_{\alpha,d} \) equipped with the compact-open topology.
Given any ordinal \( \alpha \), Gheysens \cite{GheysensDynamics,GheysensHomeomorphism} has proven that \( \Homeo(\alpha) \), equipped with the topology of pointwise convergence, is a Roelcke precompact topological group, implying that it is coarsely bounded. 

Ideally, we would like to use  \Cref{thm:ses} to compute the abelianizations of \( \PH_{\alpha,d} \) and \( \HH_{\alpha,d} \), as well as to produce normal generating sets. 
The main obstacle is working with \( \F_{\alpha,d} \), which is a complicated group.  
If we quotient out by \( \F_{\alpha,d} \), then we do have the tools to proceed.
The downside is that the isomorphism class of this quotient does not depend on \( \alpha \).
Indeed, let \( \mathrm K_{\alpha,d} \) denote the kernel of the action of \( \HH_{\alpha,d} \) on the set of capacity \( \alpha \) elements of \( X_{\alpha,d} \). 
It can be readily verified that \( \HH_{\alpha,d} / \mathrm K_{\alpha,d} \) is isomorphic to \( \HH_{0,d} \), leading us to the investigation of \( \HH_{0,d} \). 
The difference in the case \( \alpha =0 \) is that \( \F_{0,d} \) is the group of finitely supported homeomorphisms, as \( X_{0,d} \) has exactly \( d \) accumulation points. 
In particular, the elements of \( \F_{0,d} \) are supported away from elements of maximal capacity.

\begin{MainThm5}
\label{thm:d}
Let  \( d \in \bn \).
For  \( f \in \overline \F_{0,d} \), the following are equivalent:
	\begin{enumerate}[(i)]
	
		\item \( f \) normally generates \( \overline \F_{0,d} \).
		
		\item \( f \) uniformly normally generates \( \overline F_{0,d} \).
		
		\item \( f \) induces an infinite permutation of the isolated points in any open neighborhood of an accumulation point. 
	\end{enumerate}
	
\end{MainThm5}

The actual \Cref{thm:normal generators F} stated in \Cref{sec:F0} contains another equivalent condition---related to \Cref{thm:normal generators 2}---that we omit here, as we do not have the appropriate definitions yet at hand.
The proof of \Cref{thm:normal generators F} follows from a  fragmentation lemma, \Cref{lem:factorization2}, that holds for all \( \alpha \) but that is strongest when \( \alpha =0 \), again as \( \F_{0,d} \) consists only of finitely supported homeomorphisms.  
Fragmentation will allow us to view any element of \( \overline \F_{0,d} \) as a composition of \( d+1 \) homeomorphisms, each of which can be viewed as homeomorphism of \( X_{0,1} \).  
This fragmentation also allows us to consider the abelianization of \( \overline \F_{0,d} \).

\begin{MainThm5'}
If \( d \in \bn \), then \( \overline \F_{0,d} \) is uniformly perfect and has commutator width at most four. 
\end{MainThm5'}

As consequences of \Cref{thm:ses}, \Cref{thm:perfect F}, and \Cref{thm:normal generators F}, we can compute the abelianizations of \( \PH_{0,d} \) and \( \HH_{0,d} \) as well as the minimal cardinality of normal generating sets.
A subset of a group is a \emph{normal generating set} if the elements of the set together with all their conjugates generate the group.

\begin{MainThm6}

    If \( d \in \bn\ssm\{1\} \), then:
    
    	\begin{enumerate}[(1)]
    	
    		\item The abelianization of \( \PH_{0,d} \) is isomorphic to \( \bz^{d-1} \).
    		
    		\item The abelianization of \( \HH_{0,d} \) is isomorphic to \( \bz/2\bz \times \bz/2\bz \).
    		
    	\end{enumerate}
	
\end{MainThm6}

The computation of the abelianization of \( \PH_{0,d} \) and of \( \HH_{0,d} \) implies that the minimal cardinality for a normal generating set is at least \( d-1 \) and at least two, respectively.
Our final theorem, says this is in fact an equality, and the proofs construct explicit normal generating sets of the desired cardinalities. 

\begin{MainThm7}

	If \( d \in \bn \ssm \{1\} \), then the minimal cardinality of a normal generating set for \( \PH_{0,d} \) is \( d-1 \) and is two for \( \HH_{0,d} \).

\end{MainThm7}

We do not have a good sense of whether one should expect the above theorems to hold for \( \alpha > 0 \).
A resolution in either direction would be interesting.

\subsection*{Connections to geometric topology}

As mentioned at the beginning of the introduction, the results above were motivated by recent work in the setting of homeomorphism groups of surfaces.
We present this connection here, and also contextualize our results in the surface setting. 

The end space of a non-compact manifold is an encoding of the different ways of escaping to infinity in the manifold, and this set of directions naturally inherits a topology from the manifold.
For example, \( \br \) is two-ended while \( \br^2 \) is one ended, and if a compact totally disconnected subset \( E \) of the 2-sphere \( \mathbb S^2 \) is removed from \( \mathbb S^2 \), then the end space of the resulting manifold, \( \mathbb S^2 \ssm E \), is homeomorphic to \( E \). 
The end space is a topological invariant of the manifold and hence any homeomorphism between manifolds induces a homeomorphism on the end spaces. 
Therefore, given a manifold \( M \) with end space \( E \) there is a canonical homomorphism \( \Homeo(M) \to \Homeo(E) \) given by the action of \( \Homeo(M) \) on \( E \). 

Recall that a \emph{Stone space} is a compact, zero-dimensional, Hausdorff topological space. 
If \( E \) is a second-countable Stone space, then \( E \) embeds as a subset of \( \mathbb S^2 \).
Identifying \( E \) with such an embedding, it follows from \cite{RichardsClassification} that the canonical homomorphism \( \Homeo(\mathbb S^2 \ssm E) \to \Homeo(E) \) is surjective. 
Moreover, as any two isotopic homeomorphisms induce the same action on \( E \), this homomorphism factors to give an epimorphism \( \mcg(\mathbb S^2 \ssm E) \to \Homeo(E) \), where \( \mcg(\mathbb S^2\ssm E) \) is the group of isotopy classes of homeomorphisms \( \mathbb S^2 \ssm E \to \mathbb S^2 \ssm E \), known as the \emph{mapping class group}. 
If \( E \) is infinite, the group \( \mcg(\mathbb S^2 \ssm E) \) is an example of a \emph{big} mapping class group, a class of groups that have been investigated intensely recently. 
We therefore see that theorems about big mapping class groups yield corollaries for homeomorphism groups of second-countable Stone spaces. 

Homeomorphism groups of Stone spaces have a long history of being studied from the perspective of automorphism groups of Boolean algebras, which are dual to Stone spaces by Stone duality.
The mapping class group perspective has brought a renewed interest.
The compact ordinals are Stone spaces, and the topological structure of ordinals are simpler than Stone spaces in general. 
This has led to a recent interest in homeomorphism groups of ordinals, e.g., see \cite{HernandezAmple, BestvinaClassification, BranmanGraphical}.


Let us finish with exploring how our work has implications for mapping class groups.
For a countable  ordinal \( \alpha \), let us identify \( X_{\alpha,1} \) with a subset of \( \mathbb S^2 \).  
Our first slate of results fail for \( \mcg(\mathbb S^2 \ssm X_{\alpha,1}) \). 
Let us explain:
Using the work of Domat--Dickmann \cite{DomatBig}, Malestein--Tao \cite{MalesteinSelf} showed that \( \mcg( \mathbb S^2 \ssm X_{0,1} ) \) surjects onto the additive group of real numbers.
Now, there is an epimorphism \( \mcg(\mathbb S^2 \ssm X_{\alpha,1}) \to \mcg(\mathbb S^2 \ssm X_{0,1}) \) obtained by ``filling in'' the ends whose capacity is less than \( \alpha \) (this is called a \emph{forgetful homomorphism} in the literature). 
It follows that \( \mcg(\mathbb S^2 \ssm X_{\alpha,1}) \) also surjects onto the reals; in particular, \( \mcg(\mathbb S^2 \ssm X_{\alpha,1}) \) is neither uniformly perfect,  finitely normally generated, strongly bounded, nor strongly distorted. 
Therefore, our results together with the work of Malestein--Tao, show that the failure of \( \Homeo(\mathbb S^2 \ssm X_{\alpha,1}) \) to be strongly bounded, uniformly perfect, and finitely normally generated, stems from a two-dimensional obstruction, i.e., it cannot be detected by the action on the end space.
It is also interesting to note that the results only fail in the category of abstract groups, meaning their analogs in the category of topological groups hold; in particular, Mann--Rafi \cite{MannLarge-scale} showed that \( \mcg(\mathbb S^2 \ssm X_{\alpha,1}) \) is coarsely bounded, and it follows from the work of the last author with Lanier \cite{LanierMapping} that \( \mcg(\mathbb S^2 \ssm X_{\alpha,1}) \) is topologically normally generated by a single mapping class (Baik \cite{BaikTopological} also showed finite topological normal generation for \( \mcg(\mathbb S^2 \ssm X_{\alpha,d}) \), where \( \alpha \) is a countable ordinal and \( d \in \bn \)).

\subsection*{Acknowledgments}
This paper stems from a vertically integrated summer research program at Queens College supported by NGV's NSF Award DMS-2212922.
NGV was also supported in part by PSC-CUNY Awards  66435-00 54 and Award 67380-00 55. 

The authors thank Kathryn Mann for the reference to Galvin's work and Ferr\'an Valdez for a discussion of his forthcoming work with collaborators on the geometry of homeomorphism groups of ordinals. 
We also thank the referee who provided a number of invaluable comments that have improved our exposition immensely. 


\section{Ordinals and their topology}
\label{sec:ordinals}

We provide an introduction to ordinals and their topology, aimed at non-experts. 
In the first three subsections we cover the basics of ordinals, see for instance \cite[Chapter~2]{LevyBasic}, \cite[\S4.1--2]{CiesielskiSet}, and \cite[Chapter~2]{JechSet} for the set-theoretic aspects. 

\subsection{Defining the ordinals}

A binary relation \( \leq \) on a set \( X \) is a \emph{(strict) total order} if:
	\begin{itemize}

		\item it is \emph{reflexive}: \( x \leq x \),
		\item it is \emph{transitive}: if \( x \leq y \) and \( y \leq z \) then \( x \leq z \),
		\item it is \emph{antisymmetric}: if \( x \leq y \) and \( y \leq x \) then \( x = y \), and
		\item it is \emph{strongly connected}: \( x \leq y \) or \( y \leq x \)
	\end{itemize}
for all \( x,y,z \in X \).
If, in addition, for any non-empty subset \( A \) of \( X \) there exists \( a \in A \) such that \( a \leq b \) for all \( b \in A \), then \( \leq \) is \emph{well-ordered}. 
A well-ordered set is a pair \( (X, \leq) \). 
Two well-ordered sets \( X \) and \( Y \) are \emph{order isomorphic} if there exists an order-preserving bijection \( X \to Y \). 
The strict total order \( \leq \) gives rise to a non-strict order \( < \), which is defined by \( x < y \) if and only if \( x \leq y \) and \( x \neq y \). 

\begin{Def}[Ordinal]
	An \emph{ordinal} is the order-isomorphism class of a well-ordered set.
\end{Def}

This definition has the advantage of being easy to state and convenient for introducing ordinal arithmetic.
However, in general, it will be immensely helpful to have a canonical representative of an ordinal, which we now describe. 

We will use von Neumann's definition of the ordinals.

\begin{Def}[von Neumann ordinal]
A set \( \alpha \) is an \emph{von Neumann ordinal} if it satisfies the following:
\begin{enumerate}[(i)]
	\item if \( \beta \in \alpha \), then \( \beta \subset \alpha \);
	\item if \( \beta, \gamma \in \alpha \), then \( \beta = \gamma \), \( \beta \in \gamma \), or \( \gamma \in \beta \); and
	\item if \( B \subset \alpha \) is nonempty, then there exists \( \gamma \in B \) such that \( \gamma \cap B = \varnothing \). 
\end{enumerate}
\end{Def}


Every element of a von Neumann ordinal is itself a von Neumann ordinal.
A von Neumann ordinal can be well-ordered by set inclusion, i.e., if \( \alpha \) is a von Neumann ordinal and \( \beta, \gamma \in \alpha \), declaring \( \beta \leq \gamma \) if and only if \( \beta \subseteq \gamma \) yields a well-ordering relation on \( \alpha \).
Note that if \( \beta, \gamma \in \alpha \), then \( \gamma < \beta \) if and only if \( \gamma \in \beta \). 
In particular, if \( \beta \in \alpha \), then \( \beta = \{ \eta \in \alpha : \eta < \beta \} \). 
If \( \alpha \) and \( \beta \) are order-isomorphic von Neumann ordinals, then  \( \alpha = \beta \); it now follows from basic properties of well-ordered sets that given two von Neumann ordinals \( \alpha \) and \( \beta \), either \( \alpha = \beta \), \( \alpha \in \beta \), or \( \beta \in \alpha \). 
We can therefore compare any two distinct von Neumann ordinals, writing \( \beta < \alpha \) if and only if \( \beta \in \alpha \) (here, we have slightly abused notation, using \( < \) to denote both a relation on a given von Neumann ordinal and as a comparison of two distinct von Neumann ordinals; however, no confusion arises with this abuse).
In other words,
\begin{quote}
	\emph{The von Neumann ordinal \( \alpha \) is the set of all von Neumann ordinals less than \( \alpha \).}
\end{quote}

Now, every well-ordered set is order-isomorphic to a unique von Neumann ordinal, implying that a von Neumann ordinal is a canonical representative of its order-isomorphism class. 

\textbf{Notation.}
	Throughout the paper, we will abuse notation and conflate an ordinal with its associated von Neumann ordinal, using the word ordinal to refer to both simultaneously.
	It should clear from context whether we are using an ordinal as an isomorphism class or working directly with its von Neumann representative.
	We will write ordinals using lower-case Greek letters.

Given any ordinal \( \alpha \), the set \( \alpha \cup \{\alpha\} \) is an ordinal, called the \emph{successor of \( \alpha \)}. 
An ordinal is a \emph{successor ordinal} if it is the successor to some ordinal; otherwise, it is a \emph{limit ordinal}\footnote{We follow \cite{CiesielskiSet} in considering 0 a limit ordinal.}.
From the discussion above, any set of ordinals has a least element; this least element is necessarily the empty set and is denoted by 0.
Therefore, 0 is the first limit ordinal. 
We can now proceed to construct the natural numbers, identifying 1 with \( \{0\} \), or equivalently, \( \{\varnothing\} \), 2 with \( \{ 0, 1\} \), or equivalently, \( \{ \varnothing, \{\varnothing\} \} \), etc. 
Throughout, we will identity the natural numbers, denoted \( \bn \), with \( \bn = \{ 1, 2, \ldots \} \); in particular, \( 0 \notin \bn \). 
The natural numbers together with 0 constitute the set of finite ordinals, which itself is an ordinal, denoted \( \omega \); equivalently, \( \omega \) is the first infinite ordinal.  
Note that \( \omega \) is the first non-zero limit ordinal. 

Let us finish this subsection with recalling transfinite induction, which follows from the existence of minimal elements in well-ordered sets. 

\begin{Thm}[Transfinite induction]
	Let \( \alpha \) be an ordinal, and let \( P(\beta) \) be a property defined for all \( \beta \leq \alpha \).
	If
	\begin{enumerate}[(i)]
		\item (base case) \( P(0) \),
		\item (successor case) \( P(\beta) \) implies \( P(\beta+1) \), and 
		\item (limit case) for a limit ordinal \( \lambda \),  \( P(\beta) \) for all \( \beta < \lambda \) implies \( P(\lambda) \).
	\end{enumerate}
	then \( P(\alpha) \). 
	\qed
\end{Thm}


\subsection{Ordinal arithmetic}

We now turn to defining ordinal arithmetic, including the notions of addition, multiplication, and exponentiation. 

\begin{Def}[Ordinal addition]
\label{def:addition}
Let \( \alpha \) and \( \beta \) be ordinals. 
We define the \emph{sum} of \( \alpha \) and \( \beta \), denoted \( \alpha + \beta \), to be the order-isomorphism class of the set \( \alpha \sqcup \beta \) equipped with the total order that extends the orders on \( \alpha \) and \( \beta \) and declares the elements of \( \alpha \) to be smaller than those of \( \beta \), that is, \( \alpha \sqcup \beta \) equipped with the order \( \preceq \) given by \( \eta \preceq \gamma \) if and only if either
\begin{itemize}
	\item \( \eta, \gamma \in \alpha \) and \( \eta \leq \gamma \),
	\item \( \eta, \gamma \in \beta \) and \( \eta \leq \gamma \), or
	\item \( \eta \in \alpha \) and \( \gamma \in \beta \). 
\end{itemize}
\end{Def}

Observe that \( \alpha +1 \) is the successor of \( \alpha \) for any ordinal \( \alpha \). 
Note that addition of ordinals is not commutative, e.g., \( 1 + \omega \neq \omega +1 \), as \( 1 + \omega \) is a limit ordinal (namely \( \omega \)) and \( \omega+1 \) is a successor ordinal. 
However, ordinal addition is associative. 

\begin{Def}[Ordinal multiplication]
\label{def:multiplication}
We define the \emph{product} of \( \alpha \) and \( \beta \), denoted \( \alpha\cdot\beta \), to be the order-isomorphism class of the set \( \beta \times \alpha \) equipped with the lexicographical ordering. 
\end{Def}

As for summation, multiplication of ordinals is not commutative, e.g., \( 2 \cdot \omega \neq \omega \cdot 2 \), as the elements of \( 2 \cdot \omega \) are all finite ordinals (i.e., \( 2\cdot \omega = \omega \)) whereas \( \omega \in \omega \cdot 2 \). 

Though commutativity of addition and multiplication both fail, it is readily verified that associativity and left-distributivity hold.
We record this fact in the following lemma.

\begin{Lem}\label{lem:associative}
	Let \( \alpha, \beta \), and \( \gamma \) be ordinals. 
	Then 
		\[ (\alpha+\beta)+\gamma = \alpha + (\beta+\gamma) \]
	and
		\[ \alpha \cdot (\beta + \gamma) = \alpha \cdot \beta + \alpha \cdot \gamma .\]
	\qed
	\end{Lem}

\begin{Def}[Ordinal exponentiation]
We define the \emph{exponentiation} of \( \alpha \) by \( \beta \), denoted \( \alpha^\beta \), to be the ordinal that is order isomorphic to the set of finitely supported functions 
\[ 
	\{ f \co \beta \to \alpha : f(\eta) = 0 \text{ for all but finitely many } \eta \}
\]
equipped with the order \( \prec \) given by \( f \prec g \) if and only if there exists \( \eta \in \beta \) such that \( f(\eta) < g(\eta) \) and \( f(\gamma) = g(\gamma) \) for all \( \gamma < \eta \). 
\end{Def}

All the above ordinal arithmetic operations can be defined using transfinite recursion.
It is worth noting how this works for exponentiation: we set \( \alpha^0 = 1 \), \( \alpha^{\beta+1} = \alpha^\beta \cdot \alpha \) for all \( \beta \), and \( \alpha^\lambda = \bigcup_{\beta<\lambda} \alpha^\beta \) for a limit ordinal \( \lambda \). 
We mention this as it is important for us to know that
\begin{equation}
\label{eq:exponent}
	\omega^{\alpha+1} = \omega^\alpha \cdot \omega .
\end{equation}

We can now provide a normal form for ordinals, known as Cantor normal form:

\begin{Thm}[Cantor Normal Form]
If \( \alpha \) is a non-zero ordinal, then \( \alpha \) can be uniquely represented in the form
\begin{equation}
	\label{eq:cantor}
	\alpha = \omega^{\beta_1}\cdot k_1 + \cdots + \omega^{\beta_n}\cdot k_n,
\end{equation}
where \( n \in \bn \), \( \alpha \geq \beta_1 > \cdots > \beta_n \), and \( k_1, \ldots, k_n \in \bn \). 
\qed
\end{Thm}

Following \cite{KieftenbeldClassification}, we call \( \beta_1 \) in \eqref{eq:cantor} the \emph{limit capacity} of \( \alpha \) and \( k_1 \) the \emph{coefficient} of \( \alpha \). 

\subsection{Topology of ordinals}

An ordinal has a canonical topology given by the order topology; for all that follows, whenever we reference the topology of an ordinal, we are referring to the order topology. 
Let \( \alpha \) be an ordinal.
A subbasis for the order topology on \( \alpha \) is given by the sets of form \( [0, \beta) = \{ \eta \in \alpha : \eta < \beta \} \) and \( (\beta, \alpha) = \{ \eta \in\alpha : \beta < \eta \} \), where \( \beta \in \alpha \). 

Now, let \( \lambda \in \alpha \) be a limit ordinal. 
Note that if \( \beta < \lambda \) then \( \beta+1 < \lambda \), and observe that if \( \lambda \in \alpha \) is a limit ordinal, then \( [0, \lambda) = \bigcup_{\eta < \lambda} [0, \eta+1) \).
It follows that the sets of the form \( [0, \beta+1) = [0,\beta] \) and \( (\beta, \alpha)=[\beta+1, \alpha) \) with \( \beta \in \alpha \) form a subbasis for the order topology. 
In particular, the order topology admits a basis consisting of clopen sets, i.e., it is \emph{zero-dimensional}.
Moreover, it is an exercise to check that sets of the form \( [\beta, \gamma] \) are compact and that if \( \lambda \) is a limit ordinal then \( [0, \lambda) \) fails to be compact. 
We summarize this in the following proposition. 

\begin{Prop}
An ordinal is a zero-dimensional Hausdorff topological space.
Moreover, a non-zero ordinal is compact if and only if it is a successor ordinal. 
\qed
\end{Prop}

The next goal of this subsection is to classify the successor ordinals up to homeomorphism, which is accomplished in \Cref{thm:classification}.
In the process, we establish several preliminary results and introduce the Cantor--Bendixson derivative.

\begin{Rem}
A topological classification of ordinals, both successor and limit, is given by Kieftenbeld--L{\"o}we in \cite{KieftenbeldClassification}, a report for the Institute for Logic, Language, and Computation.
We reproduce a complete argument for the classification of the compact ordinals here as it is instructive and convenient for the reader.
\end{Rem}

\begin{Lem}
\label{lem:commutative}
If \( \alpha \) and \( \beta \) are successor ordinals, then \( \alpha + \beta \) is homeomorphic to the topological disjoint union of \( \alpha \) and \( \beta \).
In particular, \( \alpha + \beta \) and \( \beta + \alpha \) are homeomorphic. 
\end{Lem}

\begin{proof}
By definition, the ordinal \( \alpha + \beta \) is represented by the disjoint union of \( \alpha \) and \( \beta \) with the order \( \preceq \) given in Definition~\ref{def:addition}.
In the space \( \alpha + \beta \), \( \beta \) is closed and \( \alpha \) is open. 
As \( \alpha \) is compact, it must also be closed.
Therefore, both \( \alpha \) and \( \beta \) are clopen subsets of \( \alpha + \beta \), implying \( \alpha + \beta \) is topologically the disjoint union of \( \alpha \) and \( \beta \).
Similarly, \( \beta + \alpha \) is homeomorphic to the topological disjoint union of \( \alpha \) and \( \beta \), implying that \( \alpha + \beta \) and \( \beta + \alpha \) are homeomorphic. 
\end{proof}

\begin{Lem}
\label{lem:sum-equal}
Let \( \alpha \) and \( \beta \) be ordinals. 
If the limit capacity of \( \beta \) is strictly less than that of \( \alpha \), then \( \beta + \alpha = \alpha \). 
\end{Lem}

\begin{proof}
As the limit capacity of \( \beta \) is strictly less than \( \alpha \), it follows that \( \beta \cdot \omega \leq \alpha \), as \( \beta + k < \alpha \) for all \( k \in \omega \). 
Let \( \gamma \) be the ordinal of  the interval \( (\beta\cdot \omega, \alpha) \) in \( \alpha \).
Then \( \alpha = \beta \cdot \omega + \gamma \). 
Using associativity and distributivity (\Cref{lem:associative}), we have:
	\[ \beta + \alpha = \beta + (\beta\cdot \omega + \gamma) = (\beta + \beta\cdot \omega)+\gamma = \beta\cdot(1+\omega)+\gamma = \beta\cdot \omega + \gamma = \alpha \]
as desired.
\end{proof}

\begin{Prop}
\label{prop:reverse}
Let \( \alpha \) be a successor ordinal. 
If \( \beta \) is the limit capacity of \( \alpha \) and \( k \) is its coefficient, then \( \alpha \) is homeomorphic to \( \omega^\beta \cdot k + 1 \). 
\end{Prop}

\begin{proof}
As \( \alpha \) is a successor, the Cantor normal form of \( \alpha \) is given by 
\[
	\alpha = \omega^{\beta_1}\cdot k_1 + \cdots + \omega^{\beta_n} \cdot k_n + 1,
\]
where \( \beta_1 = \beta \) and \( k_1 = k \). 
We can therefore write
\[
	\alpha = \left(\omega^{\beta}\cdot k + 1 \right) + \gamma,
\]
where \( \gamma \) is a successor ordinal. 
By \Cref{lem:commutative}, \( \alpha \) is homeomorphic to \( \gamma + \left(\omega^{\beta} \cdot k + 1\right) \). 
As the limit capacity of \( \gamma \) is strictly less than that of \( \omega^{\beta} \cdot k + 1 \), \Cref{lem:sum-equal} implies that \( \gamma + \left(\omega^{\beta} \cdot k + 1\right) = \omega^{\beta} \cdot k + 1 \), and hence \( \alpha \) is homeomorphic to \( \omega^{\beta}\cdot k + 1 \). 
\end{proof}

To classify the compact ordinals, we need to show that the limit capacity and the coefficient of a successor ordinal are topological invariants. 
For this, we need the Cantor--Bendixson derivative. 

\begin{Def}[Cantor--Bendixson derivative]
Let \( X \) be a topological space.
The \emph{derived set of \( X \)}, denoted \( X' \), is the subset of \( X \) consisting of all the accumulation points in \( X \). 
Given an ordinal \( \alpha \), the \( \alpha^{\text{th}} \)-Cantor--Bendixson derivative of \( X \), denoted \( X^{(\alpha)} \), is defined via transfinite recursion as follows:
\begin{itemize}
	\item \( X^{(0)} = X \),
	\item \( X^{(\alpha+1)} = \left(X^{(\alpha)}\right)' \), and
	\item if \( \lambda \) is a limit ordinal, then \( X^{(\lambda)} = \bigcap_{\beta < \lambda} X^{(\beta)} \). 
\end{itemize}
\end{Def}

Next, we introduce the Cantor--Bendixson rank and degree.

\begin{Def}[Cantor--Bendixson rank]
Let \( X \) be a topological space. 
The \emph{Cantor--Bendixson rank} of \( X \) is the least ordinal \( \alpha \) such that \( X^{(\alpha)} \) is perfect. 
\end{Def}

As the intersection of nonempty nested closed sets is nonempty in a compact space, if \( X \) is a compact space whose Cantor--Bendixson derivatives stabilize in the empty set, then its Cantor--Bendixson rank is necessarily a successor ordinal and the last non-empty Cantor--Bendixson derivative has finite cardinality, allowing us to define the Cantor--Bendixson degree.

\begin{Def}[Cantor--Bendixson degree]
	Let \( X \) be a compact topological space whose Cantor--Bendixson derivatives are eventually empty.
	If \( \alpha + 1 \) is the Cantor--Bendixson rank of \( X \), then the \emph{Cantor--Bendixson degree} of \( X \) is the cardinality of \( X^{(\alpha)} \). 
%
\end{Def}

A successor ordinal is compact, non-perfect, and its Cantor--Bendixson derivatives are eventually empty.
Therefore, we can investigate both is Cantor--Bendixson rank and degree.
Using the Cantor normal form and transfinite induction, it is an exercise to show the following:

\begin{Prop}
\label{prop:forward}
The Cantor--Bendixson rank of a successor ordinal is equal to the successor of its limit capacity and its Cantor--Bendixson degree is equal to its coefficient.
\qed
\end{Prop}

The classification of successor ordinals up to homeomorphism now follows immediately from \Cref{prop:reverse} and \Cref{prop:forward}.

\begin{Thm}[Topological classification of successor ordinals]
\label{thm:classification}
Two successor ordinals are homeomorphic if and only if their limit capacities and their coefficients agree, or equivalently, their Cantor--Bendixson ranks and degrees agree. 
\qed
\end{Thm}

Next, we establish a stability property for neighborhoods of elements of ordinals.
First, we need to extend the notion of Cantor--Bendixson rank to elements of ordinals.

\begin{Def}[Rank of a point]
Let \( X \) be a topological space of Cantor--Bendixson rank \( \alpha \) satisfying \( X^{(\alpha)} = \varnothing \). 
For \( x \in X \), we define the \emph{rank of \( x \)} to be the least ordinal \( \beta \) such that \( x \notin X^{(\beta)} \). 
\end{Def}

Observe that the rank of a point is necessarily a successor ordinal. 
By \Cref{prop:forward}, the rank of an element in an ordinal is the successor to its capacity. 
Despite discussing the capacity of a point in the introduction, we will generally discuss the rank going forward, as it is topologically defined. 

We can use the classification of successor ordinals to show that each element of an ordinal is \emph{stable} in the sense introduced by Mann--Rafi \cite{MannLarge-scale}.

\begin{Def}[Stable point]
Let \( X \) be a zero-dimensional topological space.
A point \( x \in X \) is \emph{stable} if it admits a neighborhood basis consisting of pairwise homeomorphic clopen sets. 
\end{Def}

\begin{Lem}
\label{lem:stable}
Let \( \alpha \) be an ordinal,  let \( b \in \alpha +1 \), let \( \beta+1 \) be the rank of \( b \), and let \( U \) be a clopen neighborhood of \( b \) in \( \alpha + 1 \). 
If the rank of every element of \( U \ssm \{b\} \) has rank strictly less than \( \beta + 1 \), then \( U \) is homeomorphic to \( \omega^{\beta} + 1 \). 
\end{Lem}

\begin{proof}
Observe that if \( \eta \leq \beta \), then \( U \) contains an element of rank \( \eta \). 
It follows that, as \( U \) is a well-ordered set, \( U \) must be isomorphic to an ordinal of Cantor--Bendixson rank \( \beta + 1 \) and Cantor--Bendixson degree 1. 
Therefore, by \Cref{thm:classification}, \( U \) homeomorphic to \( \omega^\beta + 1 \). 
\end{proof}

Note that all sufficiently small neighborhoods of \( b \) in \( \alpha \) satisfy the hypothesis of \Cref{lem:stable}, allowing us to conclude that \( b \) is stable. 

\begin{Prop}
Every element of an ordinal is stable. 
\qed
\end{Prop}

As already noted, \Cref{lem:stable} says that all sufficiently small clopen neighborhoods of \( b \) in \( \alpha \) are homeomorphic. 
It will be useful to have a name for such sets.

\begin{Def}[Stable neighborhood]
For an ordinal \( \alpha \) and \( b \in \alpha + 1 \), a clopen neighborhood \( U \) of \( b \) in \( \alpha + 1 \) is \emph{stable} if \( b \) is the unique highest rank element in \( U \).
\end{Def}

By \Cref{lem:stable}, a stable neighborhood of a rank \( \beta + 1 \) element of \( \alpha + 1 \) is homeomorphic to \( \omega^\beta + 1 \).

Our theorems deal with the successor ordinals whose limit capacities are successor ordinals, that is, ordinals that are homeomorphic to \( \omega^{\alpha + 1}\cdot d + 1 \). 
For the remainder of this section, we focus on \( d =1 \), introducing a convenient coordinate system.
For any ordinal \( \alpha \), the successor of \( \omega^\alpha \) is \( \omega^\alpha \cup \{ \omega^\alpha \}  = \omega^\alpha + 1 \); topologically, this says that \( \omega^\alpha + 1 \) is the one-point compactification of \( \omega^\alpha \). 
Moreover, \( \omega^{\alpha} \) is the unique rank \( \alpha + 1 \) element of \( \omega^{\alpha} + 1 \), which means that every homeomorphism of \( \omega^\alpha +1 \) fixes \( \omega^\alpha \), which yields the following.

\begin{Lem}
\label{lem:isomorphism}
For any ordinal \( \alpha \),  \( \Homeo(\omega^\alpha + 1) \) is isomorphic to \( \Homeo(\omega^\alpha) \).
\qed 
\end{Lem}

In the next section, where we focus on Cantor--Bendixson degree one ordinals, the prior lemma allows us to focus on \( \omega^\alpha \) rather than \( \omega^{\alpha} + 1 \). 
The advantage of this viewpoint is the following proposition.

\begin{Prop}\label{prop:coordinates}
	Let \( \alpha \) be an ordinal. 
	If \( X \) is a countable discrete space, then the spaces \( \omega^{\alpha+1} \) and \( X \times (\omega^{\alpha} +1 ) \) are homeomorphic. 
\end{Prop}

\begin{proof}
For each \( k \in \omega \), \( \omega^\alpha \cdot k +1 \in \omega^{\alpha +1} \). 
Now, for \( \eta \in \omega^{\alpha +1} \), there exists \( k \in \bn \) such that \( \eta \in \omega^\alpha\cdot k + 1 \).
It follows that
\begin{align*}
	\omega^{\alpha+1}	&= \bigcup_{k\in\bn} [0,\omega^\alpha\cdot k] \\
										&= \{0\} \cup \bigcup_{k\in\omega} [\omega^\alpha\cdot k + 1, \omega^\alpha \cdot (k+1)] \\
					&=  \bigcup_{k\in \omega} I_k,
\end{align*}
where we let \( I_0 = [0, \omega^\alpha] \) and \( I_k = [\omega^\alpha\cdot k + 1, \omega^\alpha \cdot (k+1)] \) for \( k \in \bn \).
Now, for distinct \( j,k \in \omega \),  \( I_k \) is clopen and \( I_k \cap I_j = \varnothing \).
Therefore, \( \omega^{\alpha+1} \) is homeomorphic to the topological disjoint union of the \( I_k \).
To finish, observe that each of the \( I_k \) is order isomorphic to, and hence homeomorphic to, \( \omega^\alpha + 1 \), implying that \( \omega^{\alpha+1} \) is homeomorphic to \( \omega \times (\omega^\alpha + 1) \).
As as \( \omega \) is countable and discrete, it is homeomorphic to \( X \), implying \( \omega \times (\omega^\alpha + 1) \), and hence \( \omega^{\alpha+1} \), is homeomorphic to \( X \times (\omega^\alpha+1) \). 
\end{proof}

\subsection{Topology of homeomorphism groups}

Though our main focus is on homeomorphism groups as abstract groups, a key subgroup in what follows is defined by taking a closure in an appropriate topology on the homeomorphism group of an ordinal; moreover, \Cref{thm:ses} is strengthened by making a statement in the category of topological groups, rather than that of  abstract groups.  

Let us recall that a \emph{topological group} \( G \) is an abstract group equipped with a Hausdorff topology such that the multiplication map \( G \times G \to G \) given by \( (g,h) \mapsto gh \)  and the inversion map  \( G \to G \) given by \( g \mapsto g^{-1} \) are continuous.

\begin{Def}[Permutation topology]
For an ordinal \( \alpha \), the \emph{permutation topology} on the group \( \Homeo(\alpha+1) \) is generated by the sets of the form \( U(A,B) = \{ h \in \Homeo(\alpha) : h(A) = B \} \), where \( A \) and \( B \) are clopen subsets of \( \alpha +1 \). 
\end{Def}

The permutation topology is motivated by the model theory perspective and arises from the identification, via Stone duality, of \( \Homeo(\alpha+1) \) with the automorphism group of the Boolean algebra given by the clopen subsets of \( \alpha+1 \). 
It is an exercise to show that that the permutation topology agrees with the \emph{compact-open topology}, that is, the topology generated by sets of the form \( U(K,W) = \{h\in \Homeo(\alpha + 1) : h(K) \subset W \} \), where \( K \) is compact and \( W \) is open in \( \alpha  + 1 \). 
A theorem of Arens \cite{ArensTopologies}---which the reader may treat as an exercise---tells us that the homeomorphism group of a compact Hausdorff space equipped with the compact-open topology is a topological group.
From this discussion, we record the following.

\begin{Prop}
For any ordinal \( \alpha \), the permutation topology and the compact-open topology on \( \Homeo(\alpha+1) \) agree and \( \Homeo(\alpha+1) \) equipped with the compact-open topology, equivalently the permutation topology, is a topological group. 
\qed
\end{Prop}


\section{Cantor--Bendixson degree one}

Throughout this section, we study \( \Homeo(\omega^{\alpha+1}) \) instead of \( \Homeo(\omega^{\alpha+1}+1) \), but recall they are isomorphic by \Cref{lem:isomorphism}. 

Before beginning, let us recall two standard definitions.
Let \( X \) be a topological space. 
A family of subsets \( \{ Y_n \}_{n\in\bn} \) of \( X \) is \emph{locally finite} if, given any compact set \( K \), the cardinality of the set \( \{ n\in \bn : K \cap Y_n \neq \varnothing \} \) is finite. 
The \emph{support} of a homeomorphism \( f \co X \to X \) is the closure of the set \( \{  x \in X : f(x) \neq x \} \). 
Given a sequence \( \{ f_n \}_{n\in\bn} \) of homeomorphisms \( X \to X \) whose supports form a locally finite family of pairwise-disjoint sets, the infinite product of the \( f _n \), denoted \( \prod_{n\in\bn} f_n \), is the homeomorphism that agrees with \( f_n \) on the support of \( f_n \) for each \( n \in \bn \) and restricts to the identity elsewhere. 
In the compact-open topology on \( \Homeo(X) \), the infinite product can be realized as the limit of the finite products, that is, \( \prod_{n\in\bn} f_n = \lim_{n \to \infty} \prod_{k=1}^n f_k \).

\subsection{Topological moieties}

A moiety in the natural numbers refers to a subset that has the same cardinality as its complement.  
We extend the definition to ordinals of the form \( \omega^{\alpha+1} \). 

\begin{Def}
A subset \( A \) of \( \omega^{\alpha+1} \) is a \emph{topological moiety} if \( A \) is clopen and both \( A \) and \( \omega^{\alpha+1} \ssm A \) contain infinitely many rank \( \alpha +1 \) points. 
\end{Def}

We will require several properties of moieties.  

\begin{Prop}
\label{prop:moieties same}
Let \( \alpha \) be an ordinal.
Every topological moiety of \( \omega^{\alpha+1} \) is homeomorphic to \( \omega^{\alpha+1} \).
\end{Prop}

\begin{proof}
Let \( A \) be a moiety of \( \omega^{\alpha+1} \).
Then \( A \), equipped with the order inherited from \( \omega^{\alpha+1} \), is a well-ordered set, and hence order isomorphic to an ordinal. 
Let \( \{a_n\}_{n\in\bn} \) be an enumeration of the maximal rank elements of \( A \) such that \( a_n < a_m \) whenever \( n < m \). 
Let \( U_1 = [0, a_1] \cap A \), and let \( U_n = [a_{n-1}+1, a_n] \cap A \) for \( n > 1 \). 
Then \( \{U_n\}_{n\in\bn} \) is a collection of pairwise-disjoint clopen subsets satisfying \( A = \bigcup_{n\in\bn} U_n \), implying that \( A \) is homeomorphic to the topological disjoint union of the \( A_n \). 
Now, by the classification of successor ordinals (\Cref{thm:classification}), each of the \( U_n \) is homeomorphic to \( \omega^\alpha + 1 \); therefore, \( A \) is homeomorphic to \( \bn \times (\omega^\alpha+1) \), which is in turn homeomorphic to \( \omega^{\alpha+1} \) by \Cref{prop:coordinates}. 
\end{proof}

\Cref{prop:moieties same} yields the following ``change of coordinates principle''. 

\begin{Lem}
\label{lem:change of coordinates}
Let \( \alpha \) be an ordinal.
If \( A \) and \( B \) are topological moieties, then there exists \( \sigma \) in \( \Homeo(\omega^{\alpha+1}) \) such that \( \sigma(A) = B \).
Moreover, if \( A \cap B = \varnothing \), then \( \sigma \) can be chosen to be an involution supported in \( A \cup B \). 
\end{Lem}

\begin{proof}
By \Cref{prop:moieties same}, there exists a homeomorphism \( f \co A \to B \). 
Let \( C \) and \( D \) the complements of \( A \) and \( B \), respectively.
As the complement of a topological moiety is a topological moiety, there exists a homeomorphism \( g\co C \to D \).
Define \( \sigma \in \Homeo(\omega^{\alpha+1}) \) such that \( \sigma|_A = f \) and \( \sigma|_C = g \).
Then \( \sigma(A) = B \). 

Now suppose that \( A \cap B = \varnothing \).
Define \( \iota \in \Homeo(\omega^{\alpha+1}) \) by 
\[
	\iota(x) = \left\{
		\begin{array}{ll}
			f(x)		& \text{if } x \in A \\
			f^{-1}(x) 	& \text{if } x \in B \\
			x		& \text{otherwise}
		\end{array}
	\right.
\]
Then \( \iota \) is an involution supported in \( A\cup B \) mapping \( A \) onto \( B \). 
\end{proof}

To finish this subsection, we need to establish the existence of a homeomorphism that will translate a topological moiety off of itself. 

\begin{Def}
Given a topological moiety \( A \) in \( \omega^{\alpha+1} \), we call \( \vp \in \Homeo(\omega^{\alpha+1}) \) an \( A \)-translation if \( \vp^n(A) \cap \vp^m(A) = \varnothing \) for all \( n,m \in \bz \). 
If, in addition, \( \{ \vp^n(A) \}_{n\in\bz} \) is locally finite, then we say that \( \vp \) is a \emph{convergent \( A \)-translation}. 
\end{Def}

\begin{Lem}
\label{lem:translation}
Let \( \alpha \) be an ordinal.
If \( A \) is a topological moiety in \( \omega^{\alpha+1} \), then there exists a convergent \( A \)-translation in \( \vp \in \Homeo(\omega^{\alpha+1}) \). 
Moreover, \( \vp \) can be chosen to be supported in a topological moiety. 
\end{Lem}

\begin{proof}
By \Cref{prop:coordinates}, \( \omega^{\alpha+1} \) is homeomorphic to \( \bn \times (\omega^\alpha +1) \), which in turn is homeomorphic to \( \bz \times \bn^2 \times (\omega^\alpha +1) \).
Consider the subset \( A' \) of the latter space given by \( A' = \{ 0 \} \times \{ 1 \} \times \bn \times (\omega^\alpha+1) \).
Define \( \tau' \) by \( \tau'(\ell, 1,n,x) = (\ell+1,1, n, x) \) and \( \tau'(\ell, m ,n , x) = ( \ell, m ,n , x) \) whenever \( m > 1\). 
Fix a homeomorphism \( \psi \co  \bz \times \bn^2 \times (\omega^\alpha +1) \to  \omega^{\alpha+1} \). 
Then \( \psi(A') \) is a topological moiety, and by \Cref{lem:change of coordinates}, we can assume that \( \psi(A') = A \).
It follows that \( \vp = \psi \circ \tau'\circ \psi^{-1} \) is a convergent \( A \)-translation.
Moreover, as \( \tau' \) is supported in \( \bz \times \{1 \} \times \bn \times (\omega^\alpha + 1) \), we have that \( \tau \) is supported in a topological moiety. 
\end{proof}

\subsection{Galvin's lemma}

The key behind all the proofs in the remainder of the section is a uniform fragmentation result, i.e., given two topological moieties whose union is a topological moiety, we provide a way of writing any homeomorphism as a composition of three homeomorphisms each of which is supported in one of the two given moieties. 
We call this uniform fragmentation result Galvin's lemma, as it is an extension of Galvin's \cite[Lemma~2.1]{GalvinGenerating} from the setting of permutation groups to \( \Homeo(\omega^{\alpha+1}) \).
Given a topological moiety \( A \) of \( \omega^{\alpha+1} \), we let \[ F_A = \{ f \in \Homeo(\omega^{\alpha+1}) : f(a) = a \text{ for all } a \in A\}. \]

\begin{Prop}[Galvin's lemma for ordinals]
\label{prop:galvin}
Let \( \alpha \) be an ordinal, and let \( A \) and \( B \) be disjoint topological moieties in \( \omega^{\alpha+1} \). 
If \( A \cup B \) is a topological moiety, then \[ \Homeo(\omega^{\alpha+1}) = F_AF_BF_A \cup F_BF_AF_B. \] 
\end{Prop}

\begin{proof}
Fix \( h \in \Homeo(\omega^{\alpha +1}) \). 
Let \( C \) denote the complement of \( A \sqcup B \). 
As \[ C = (C \ssm h(A)) \cup (C \ssm h(B)), \] at least one of \( C \ssm h(A) \) or \( C \ssm h(B) \) is a topological moiety. 
First, suppose that \( C \ssm h(A) \) is a topological moiety. 
We can then write \( C = M_1 \cup M_2 \), where \( M_1 \) and \( M_2 \) are disjoint topological moieties and \( h(A) \cap C \subset M_1 \). 
By \Cref{lem:change of coordinates}, it is possible to choose \( f_1 \in F_A \) such that \( f_1(B \cup M_1) = C \) and \( f_1(M_2) = B \). 

Observe that \( (f_1\circ h)(A) \) is contained in \( A \cup C \) and is disjoint from a moiety contained in \( C \). 
It follows that there exists \( f_2 \in F_B \) such that \( f_2(f_1(h(A))) = A \). 
Therefore, \( f_2 \circ f_1 \circ h = f_3 \circ f_4 \), where \( f_3 \) is supported in \( A \) and \( f_4 \) is supported in the complement of \( A \). 
So, 
\[ 
	h = f_1^{-1} \circ (f_2^{-1} \circ f_3) \circ f_4,
\]
with \( f_1^{-1}, f_4 \in F_A \) and \( f_2^{-1}, f_3 \in F_B \), implying \( h \in F_A F_B F_A \). 

As noted above, if \( C \ssm h(A) \) is not a topological moiety, then \( C \ssm h(B) \) must be a topological moiety.
In this case, a similar argument establishes that \( h \in F_B F_A F_B \). 
\end{proof}

\subsection{Normal generation and uniform perfectness}

We now turn to (1) showing that \( \Homeo(\omega^{\alpha+1}) \) is a uniformly perfect group and (2) classifying the normal generators of of \( \Homeo(\omega^{\alpha+1}) \).  

Let us recall several definitions. 
A group \( G \) is \emph{perfect} if it is equal to its commutator subgroup, i.e., the subgroup generated by the set \( \{ [x,y] : x,y \in G \} \), where \( [x,y] = xyx^{-1}y^{-1} \); it is \emph{uniformly perfect} if there exists some \( k \in \bn \) such that every element of \( G \) can be expressed as a product of \( k \) commutators, and the minimum such \( k \) is called the \emph{commutator width} of \( G \).  

\begin{Lem}
\label{lem:commutator}
Let \( \alpha \) be an ordinal, and let \( h \in \Homeo(\omega^{\alpha+1}) \).
If \( h \) is supported in a topological moiety, then \( h \) can be expressed as a single commutator. 
\end{Lem}

\begin{proof}
Let \( A \) be a topological moiety that contains the support of \( h \). 
By \Cref{lem:translation}, there exists a convergent \( A \)-translation \( \tau \in \Homeo(\omega^{\alpha+1}) \). 
We now apply a standard commutator trick.
Observe that the supports of the homeomorphisms \( \tau^n \circ h \circ \tau^{-n} \) form a locally finite family of pairwise-disjoint sets, allowing us to define the infinite product \( \sigma = \prod_{n=0}^\infty (\tau^n \circ h \circ \tau^{-n}) \).
It can now readily be checked that \( h = [\sigma, \tau] \).  
\end{proof}

\begin{Thm}
	\label{thm:mainthm2}
	If \( \alpha \) is an ordinal, then \( \Homeo(\omega^{\alpha+1}) \) is uniformly perfect, and the commutator width is at most three. 
\end{Thm}

\begin{proof}
Let \( h \in \Homeo(\omega^{\alpha+1}) \). 
By \Cref{prop:galvin}, \( h = h_1 h_2 h_3 \), with \( h_i \) supported in a topological moiety. 
By \Cref{lem:commutator}, \( h_i \) is a commutator, and hence, \( h \) can be expressed as a product of three commutators. 
\end{proof}

Recall that an element \( g \) in a group \( G \) is a \emph{normal generator} if every element of \( G \) can be expressed as a product of conjugates of \( g \) and \( g^{-1} \), or equivalently, \( G \) is the smallest normal subgroup containing \( g \).
The element \( g \) is a \emph{uniform normal generator} if there exists \( k \in \bn \) such that every element of \( G \) can be expressed as a product of at most \( k \) conjugates of \( g \) and \( g^{-1} \); the \emph{\( g \)-width} of \( G \) is the minimum such \( k \). 

To classify normal generators, we use a technique that goes back to at least the work of Anderson \cite{AndersonAlgebraic}.
As with uniform perfectness, we first investigate writing a homeomorphism supported in a moiety as a bounded-length product of conjugates of a given element and its inverse; we can then use Galvin's lemma to upgrade this to express any element in the same form.

\begin{Prop}[Anderson's method for ordinals]
\label{prop:anderson}
Let \( \alpha \) be an ordinal, and let \( h \in \Homeo(\omega^{\alpha+1}) \).
If there exists a topological moiety \( A \) such that \( h(A) \cap A = \varnothing \), then each element of \( \Homeo(\omega^{\alpha+1}) \) supported in a topological moiety can be expressed as product of four conjugates of \( h \) and \( h^{-1} \). 
\end{Prop}

\begin{proof}
Let \( f \in \Homeo(\omega^{\alpha+1}) \) such that \( f \) is supported in a moiety. 
Let \( A \) be a topological moiety such that \( h(A) \cap A  = \varnothing \), and let \( B \) be a submoiety of \( A \). 
By \Cref{lem:change of coordinates}, up to conjugating \( f \), we may assume that \( f \) is supported in \( B \). 
\Cref{prop:moieties same} and \Cref{lem:translation} imply the existence of  a convergent \( B \)-translation \( \tau \) supported in \( A \). 
Define \( \psi \in \Homeo(\omega^{\alpha+1}) \) by

\[
\psi(x) = \left\{
\begin{array}{ll}
h(x)  &   \text{if } x \in A  \\
(\tau \circ h^{-1})(x) & \text{if } x \in h(A) \\
x & \text{otherwise}
\end{array}\right.
\]

Let \( \sigma = \prod_{n=0}^\infty (\tau^n \circ f \circ \tau^{-n}) \), and let \( \eta =[\sigma, h] \). 
Then one can check that \[ f = \eta \circ \psi \circ \eta \circ \psi^{-1}. \]
Informally, one sees this by noting that \( \sigma \) is ``performing \( f \) on all the forward iterates of \( B \) in A'', and hence so is \( \eta \), but \( \eta \) is also ``performing \( f^{-1} \) on all the forward iterates of \( h(B) \) in \( h(A) \)''. 
Now, \( \psi \circ \eta \circ \psi^{-1} \) restricts to the inverse of \( \eta \) on the complement of \( B \) and is the identity on \( B \) (this is on account of \( \psi|_{h(A)} = \tau \circ h^{-1} |_{h(A)} \)).
Therefore, the composition \( \eta \circ (\psi \circ \eta \circ \psi^{-1}) \) is supported on \( B \), where it agrees with \( \sigma \) and hence with \( f \). 

Expanding out the above formulation of \( f \), we have:
\begin{align*}
	f 	&= \eta \circ \psi \circ \eta \circ \psi^{-1} \\
		&= \sigma \circ h \circ \sigma^{-1} \circ h^{-1} \circ \psi \circ \sigma \circ h \circ \sigma^{-1} \circ h^{-1} \circ \psi^{-1} \\
		&= (\sigma \circ h \circ \sigma^{-1}) \circ (h^{-1}) \circ (\psi \circ \sigma \circ h \circ \sigma^{-1} \circ \psi^{-1}) \circ (\psi \circ h^{-1} \circ \psi^{-1}),
\end{align*}
implying that \( f \) can be expressed as a product of four conjugates of \( h \) and \( h^{-1} \). 
\end{proof}

\begin{Cor}
Let \( \alpha \) be an ordinal, and let \( h \in \Homeo(\omega^{\alpha+1}) \).
If there exists a topological moiety \( A \) such that \( h(A) \cap A = \varnothing \), then \( h \) uniformly normally generates \( \Homeo(\omega^{\alpha+1}) \) and the \( h \)-width of \( \Homeo(\omega^{\alpha+1}) \) is at most twelve. 
\end{Cor}

\begin{proof}
By \Cref{prop:galvin}, each element of \( \Homeo(\omega^{\alpha+1}) \) can be expressed as a composition of three homeomorphisms each of which is supported in a topological moiety.
Therefore, by \Cref{prop:anderson}, each element of \( \Homeo(\omega^{\alpha+1}) \) can be expressed as a composition of twelve homeomorphisms each of which is a conjugate of \( h \) or \( h^{-1} \). 
\end{proof}

We have established the existence of uniform normal generators for \( \Homeo(\omega^{\alpha+1}) \), but we have yet to classify all such homeomorphisms.
To do so, we will show that if a homeomorphism induces an infinite permutation of the maximal rank  elements (i.e., the rank \( \alpha+1 \) elements), then it must move some moiety off itself, implying it is a uniform normal generator. 
The elements inducing a finite permutation on the rank \( \alpha+1 \) elements is a proper normal subgroup, and hence this will yield a complete classification of the normal generators.

\begin{Lem}
\label{lem:displacement}
Let \( \alpha \) be an ordinal, and let \( h \in \Homeo(\omega^{\alpha+1}) \). 
If \( h \) induces an infinite permutation on the set of rank \( \alpha + 1 \) elements of \( \omega^{\alpha+1} \), then there exists a moiety \( A \) such that \( h(A) \cap A = \varnothing \) and \( A \cup h(A) \) is a moiety. 
\end{Lem}

\begin{proof}
Choose a rank \( \alpha + 1 \) point \( a_1 \) such that \( h(a_1) \neq a_1 \). 
By continuity, there exists a clopen neighborhood \( A_1 \) of \( a_1 \) such that \( a_1 \) is the unique rank \( \alpha + 1 \) element of \( A_1 \) and such that  \( h(A_1) \cap A_1 = \varnothing \). 
Next, choose \( a_2 \) in the complement of \( A_1 \cup h(A_1) \cup h^{-1}(A_1) \) such that \( h(a_2) \neq a_2 \), which is possible as \( h \) induces an infinite permutation on the set of rank \( \alpha \) points. 
Again by continuity, we can choose a clopen neighborhood \( A_2 \) of \( a_2 \) in the complement of \( A_1 \cup h(A_1) \cup h^{-1}(A_1) \) such that \( h(A_2) \cap A_2 = \varnothing \). 
In particular, \( h(A_1) \cap A_2 = \varnothing \) and \( h(A_2) \cap A_1 = \varnothing \).
Proceeding in this fashion, we construct a sequence of clopen sets \( \{ A_n \}_{n\in\bn} \), each containing a unique rank \( \alpha + 1 \) point, such that \( h(A_i) \cap A_j = \varnothing \) for all \( i,j \in \bn \). 
Let \( A = \bigcup_{n\in\bn} A_{2n} \). 
Then, by construction, \( A \) is a moiety, \( h(A) \cup A \) is a moiety, and \( h(A) \cap A = \varnothing \). 
\end{proof} 

\begin{Thm}
\label{thm:normal generators 1}
Let \( \alpha \) be an ordinal. 
For \( h \in \Homeo(\omega^{\alpha+1}) \), the following are equivalent:
\begin{enumerate}[(i)]
\item \( h \) normally generates   \( \Homeo(\omega^{\alpha+1}) \).
\item \( h \) uniformly normally generates \( \Homeo(\omega^{\alpha+1}) \).
\item \( h \) induces an infinite permutation of the set of maximal rank elements of \( \omega^{\alpha+1} \). 
\end{enumerate}
Moreover, if any of the above conditions are satisfied, then the \( h \)-width of \( \Homeo(\omega^{\alpha+1}) \) is at most twelve. 
\end{Thm}

\begin{proof}
It is clear that (ii) implies (i) and that (i) implies (iii). 
Now, we assume (iii), that is, \( h \) induces an infinite permutation of the set of \( \alpha \) rank points of \( \omega^{\alpha+1} \).
In this case, \Cref{lem:displacement}, implies there exists a moiety \( A \) such that \( h(A) \cap A = \varnothing \).
By \Cref{prop:anderson}, every homeomorphism supported in a topological moiety is a product of four conjugates of \( h \) and \( h^{-1} \).
We can then apply \Cref{prop:galvin} to conclude that every element of \( \Homeo(\omega^{\alpha+1}) \) is a product of twelve conjugates of \( h \) and \( h^{-1} \), establishing that (iii) implies (ii).
\end{proof}

\subsection{Strong distortion}
\label{sec:strong distortion}

A group \( G \) is \emph{strongly distorted} if there exists \( m \in \bn \) and a sequence \( \{w_n\}_{n\in\bn} \subset \bn \) such that for any sequence \( \{ g_n\}_{n\in\bn} \subset G \) there exists \( S \subset G \) of cardinality at most \( m \) such that \( g_n \in S^{w_n} \). 
Above, for \( n \in \bn \),  \( S^n = \{ s_1s_2\cdots s_n: s_i \in S\} \) consists of the elements in \( G \) that can be written as a word of length \( n \) in \( S \). 
The notion of strong distortion was introduced by Calegari--Freedman \cite{CalegariDistortion}, and they provide a technique for establishing strong distortion that we employ here. 
This technique was described in a general form by Le Roux--Mann \cite[Construction~2.3]{LeRouxStrong} that we rephrase here in the context of ordinals. 
Below, and throughout the rest of the article, given \( g, h \in G \), we let \( g^h = h^{-1} g h \). 

\begin{Prop}[{\cite[Construction~2.3]{LeRouxStrong} for ordinals}]
\label{prop:construction}
Let \( \alpha \) be an ordinal, and let \( \{h_n\}_{n\in\bn} \) be a sequence in \( \Homeo(\omega^{\alpha+1}) \) such that \( h_n \) is supported in a topological moiety \( A \).
If \( \sigma \) is a convergent \( A \)-translation supported in a topological moiety \( B \) and if \( \tau \) is a convergent \( B \)-translation, then 
\[ 
	\vp = \prod_{n,m \in \omega} h_n^{\tau^n\sigma^m} 
\]
exists and is a homeomorphism, and  
\[ 
	h_n = [\vp^{\tau^{-n}}, \sigma]. 
\]
\qed
\end{Prop}

\begin{Thm}
\label{thm:main3}
For every ordinal \( \alpha \), the group \( \Homeo(\omega^{\alpha+1}) \) is strongly distorted.
\end{Thm}

\begin{proof}
Let \( w_n = 12n+18 \) for \( n \in \bn \). 
We will show that given any sequence \( \{ h_n\}_{n\in\bn} \subset \Homeo(\omega^{\alpha+1}) \), there exists \( S \subset \Homeo(\omega^{\alpha+1}) \) of cardinality four satisfying \( h_n \in S^{w_n} \), thereby establishing that \( \Homeo(\omega^{\alpha+1}) \) is strongly distorted.  

Fix a topological moiety \( A \).
Let us first assume that each of the \( h_n \) is supported in \( A \). 
By \Cref{lem:translation}, there exists a convergent \( A \)-translation \( \sigma \) supported in a topological moiety \( B \). 
Using \Cref{lem:translation} again, there exists a convergent \( B \)-translation \( \tau \). 
Applying \Cref{prop:construction} yields a homeomorphism \( \vp \) such that \( h_n \) can be expressed as a word of length \( 4n+4 \) in the set \( \{ \vp, \sigma, \tau \} \). 

We now assume \( \{h_n\} \) is a general sequence, but we keep \( A \), \( B \), \( \sigma \), \( \tau \), and \( \vp \) as above.
Fix a topological moiety \( C \) in \( \omega^{\alpha+1} \) such that \( A \cup C \) is a topological moiety, and using \Cref{lem:change of coordinates}, choose \( \theta \in \Homeo(\omega^{\alpha+1}) \) such that \( \theta(A) = C \). 
By \Cref{prop:galvin}, we can write 
\[
h_n = h_{n,1} \circ h_{n,2} \circ h_{n,3},
\]
where either \( h_{n,i} \) or \( h_{n,i}^\theta \) is supported in \( A \). 
Therefore, by the above argument, we can write \( h_n \) as a word of length \( 12n+18 \) in the subgroup generated by the set \( S = \{ \sigma, \tau, \vp, \theta \} \), as desired.
\end{proof}


\section{Cantor--Bendixson degree greater than one}

We now turn our attention to ordinals of the form \( \omega^{\alpha+1}\cdot d +1 \), where \( d \in \bn \). 
Let us set some notations. 
Given an ordinal \( \alpha \) and \( d \in \bn \), we let:
\begin{itemize}
	\item \( X_{\alpha,d} = \omega^{\alpha+1}\cdot d + 1 \),
	\item \( \mu_k = \omega^{\alpha+1}\cdot k \in X_{\alpha,d} \) for \( 1 \leq k \leq d \),
	\item \( M_{\alpha,d} = \{\mu_1, \ldots, \mu_d\} \) be the set of maximal rank points, i.e., the rank \( \alpha+2 \) points of \( X_{\alpha,d} \), 
	\item \( N_{\alpha,d} = \{ \omega^{\alpha+1}\cdot k + \omega^\alpha \cdot \ell : k \in d, \ell \in \bn \}  \) denote the subset of next-to-maximal rank points, i.e., the rank \( \alpha + 1 \) points of \( X_{\alpha,d} \), 
	\item  \( \HH_{\alpha,d}  = \Homeo(X_{\alpha,d}) \),
	\item  \( \F_{\alpha,d} \) denote the subgroup of \( \HH_{\alpha,d} \) consisting of homeomorphisms that induce a finite permutation on the set \( N_{\alpha,d} \),
	\item  \( \overline \F_{\alpha,d} \) denote the closure of \( \F_{\alpha,d} \) in \( \HH_{\alpha,d} \), and 
	\item  \( \PH_{\alpha,d} \) denote the subgroup of \( \HH_{\alpha,d} \) fixing each of the maximal rank points.
\end{itemize}

\medskip
We will often drop the subscripts when they are clear from context. 

\subsection{Split short exact sequences}
\label{sec:ses}

The goal of this subsection is to establish two split exact sequences, one of the form 
\begin{equation}
\label{eq:split1}
	1 \longrightarrow \overline \F_{\alpha,d} \longrightarrow \PH_{\alpha,d} \longrightarrow \bz^{d-1} \longrightarrow 1
\end{equation}
and the other of the form
\begin{equation}
\label{eq:split2}
	1 \longrightarrow  \mathrm{\PH}_{\alpha,d} \longrightarrow \HH_{\alpha,d} \longrightarrow \mathrm{Sym}(d)\longrightarrow 1
\end{equation}
where \( \mathrm{Sym}(d) \) is the symmetric group on \( d \) letters. 

Let us first describe the homomorphism \( \PH_{\alpha,d} \to \bz^{d-1} \). 
Choose pairwise-disjoint clopen neighborhoods \( \Omega_1, \ldots, \Omega_{d-1} \) of \( \mu_1, \ldots, \mu_{d-1} \), respectively. 
Given \( h \in \PH \), let 
\[ 
	O_k(h) = \{ x \in \Omega_k \cap N : h(x) \notin \Omega_k \}
\] 
and let 
\[ 
	I_k(h) = \{ x \in N \ssm \Omega_k : h(x) \in \Omega_k \},
\]
that is, \( O_k(h) \) is the subset of \( N \) consisting of points that \( h \) moves out of \( \Omega_k \) by \( h \) and \( I_k(h) \) is the subset  consisting of points of \( N \) that \( h \) moves into \( \Omega_k \). 
As \( h \in \PH \) and as \( N \) is a discrete set whose closure is \( N \cup M \), continuity implies that \( O_k(h) \) and \( I_k(h) \) are each finite.
This allows us to define \( \chi_k \co \PH \to \bz \) by
\[
	\chi_k(h) = |O_k(h)| - |I_k(h)|. 
\]
It is readily verified that \( \chi_k \) is a homomorphism. 
We will now focus our attention on the homomorphism \( \chi_{\alpha,d} \co \PH_{\alpha,d} \to \bz^{d-1} \) given by 
\[
	\chi_{\alpha,d}(h) = \left(\chi_1(h), \ldots, \chi_{d-1}(h)\right).
\]
As with our other subscripts, we will simply write \( \chi \) when doing so does not cause confusion. 
A priori, the definition of \( \chi \) depends on our choice of \( \Omega_1, \ldots, \Omega_{d-1} \), but it turns out this is not the case; we will explain after proving Theorem~\ref{thm:ses}.

\begin{Lem}
\label{lem:section}
Let \( \alpha \) be an ordinal, and let \( d \in \bn \ssm\{1\} \). 
The homomorphism \( \chi_{\alpha,d} \) admits a section. 
\end{Lem}

\begin{proof}
For \( i \in \{ 1, \ldots, d-1\} \), let \( Y_i = \omega^\alpha + 1 \), and let \( \overline Y_i  = (Y_i \times \bz) \cup \{ \hat \mu_i, \hat \mu_{d,i} \} \) be the two-point compactification of \( Y_i \times \bz \) where \( (y, n) \to \hat \mu_{d,i} \) as \( n \to \infty \) and \( (y,n) \to \hat \mu_i \) as \( n \to -\infty \). 
Define the homeomorphism \( \hat s_i \co \overline Y_i \to \overline Y_i \) by \( \hat s_i(y,n) = (y, n+1) \), \( \hat s_i(\hat \mu_i) = \hat \mu_i \), and \( \hat s_i( \hat \mu_{d,i}) = \hat \mu_{d,i}  \).

Let \( Y \) be the quotient space obtained from the disjoint union of the \( \overline Y_i \) by identifying \( \hat \mu_{d,i} \) and \( \hat \mu_{d,j} \) for \( i,j \in \{1, \ldots, d-1\} \), that is, 
\[
	Y  = \left. \left(\bigsqcup_{i=1}^{d-1} \overline Y_i \right)  \middle/ \left\{ \hat \mu_{d,1}, \ldots, \hat \mu_{d,d-1} \right\} \right.
\]
Note that each of the \( \hat s_i \) descends to a homeomorphism \( s_i \co Y \to Y \).

Now, using the classification of successor ordinals, it is an exercise to show that there exists a homeomorphism \( Y \to  X \) mapping \( \hat \mu_i \) to \( \mu_i \) and the equivalence class of the \( \hat \mu_{d,i} \) to \( \mu_d \).  
Therefore, we may view the \( s_i \) as homeomorphisms of \( X \) satisfying \( s_i^n(x) \to \mu_d \) as \( n \to \infty \) and \( s_i^n(x) \to \mu_i \) as \( n \to -\infty \) for any \( x \in \supp(s_i) \ssm M \). 
By construction, the \( s_i \) pairwise commute, as the intersection of any two of their supports is a fixed point, namely \( \mu_d \).

We claim that \( \chi_i(s_i) = 1 \) and \( \chi_j(s_i) = 0 \) if \( i \neq j \). 
Let us argue the first statement.
Let \( U_i \)\footnote{We work with an arbitrary neighborhood rather than the \( \Omega_k \), as we want to  later observe that \( \chi \) does not depend on the choice of the \( \Omega_k \).} be a clopen neighborhood of \( \mu_i \) in \( Y \) disjoint from \( \mu_j \) for \( j \neq i \). 
Note that if \( j\neq i \), then \( s_i \) restricts to the identity on \( U_i \cap \overline Y_j \), allowing us to focus on  \( \widehat U_i = U_i \cap \overline Y_i \) and to work in the coordinates of \( \overline Y_i \).
Let \( N_i = \{ n \in \bz : (\omega^\alpha, n) \in \widehat U_i \} \), and let \( a = \min (\bz \ssm N_i) - 1 \). 
Let \( b \in N_i \) such that \( \hat s_i(\omega^\alpha, b) \notin \widehat U_i \).
If \( b > a \), then letting \( c = \max \{ m \in \bz \ssm N_i : m < b \} \), we have that \( \hat s_i(\omega^\alpha, c) \in \widehat U_i \). 
With the exception of \( (\omega^\alpha, a) \), we have paired each element of the form \( (\omega^\alpha, n) \in U_i \) that \( \hat s_i \) sends out of \( \widehat U_i \) with an element of the same form in the complement of \( U_i \) that \( \hat s_i \) maps into \( \widehat U_i \). 
As \( \hat s_i(\omega^\alpha, n) = (\omega^\alpha, n+1) \in \widehat U_i \) for \( m < a \) and \( \hat s_i(\omega^\alpha, a) \notin \widehat U_i \), we can conclude  that \( \chi_i(s_i) = 1 \). 
A similar line of argument can be used to establish \( \chi_j(s_i) = 0 \) whenever \( i \neq j \). 

Letting \( \{e_1, \ldots, e_{d-1} \} \) be the standard free generating set for \( \bz^{d-1} \), it follows that the map \( \bz^{d-1} \to \PH \) given by \( e_i \mapsto s_i \) is a section of \( \chi \). 
\end{proof}

To establish \eqref{eq:split1}, we now need to show that \( \overline \F \) is the kernel of \( \chi \). 
We first establish that \( \overline \F \) is contained in the kernel and then establish the converse; we split the proof into two lemmas. 

\begin{Lem}
\label{lem:in kernel}
Let \( \alpha \) be an ordinal, and let \( d \in \bn \). 
Then \( \overline \F_{\alpha,d} \) is contained in \( \ker \chi_{\alpha,d} \). 
\end{Lem}

\begin{proof}
It is an exercise to check that \( \F \) is contained in the kernel of \( \chi \). 
Now, let \( f \in \overline \F \). 
The set
\[ 
	\mathcal U := \bigcap_{k =1}^{d-1} \left\{ h \in \HH : h(\Omega_k) = f(\Omega_k) \right\} 
\]
is an open neighborhood of \( f \) in \( \HH \). 
Therefore, as \( \F \) is dense in \( \overline \F \) by definition, there exists \( g \in \mathcal U \cap \F  \).
It follows that \( (g^{-1}\circ f)(\Omega_k) = \Omega_k \) for each \( k \), implying that 
\[ 0 = \chi(g^{-1}\circ f) = \chi(f) - \chi(g) = \chi(f), \]
as \( g \in \ker \chi \).  
\end{proof}

In what follows, we let \( \Omega_d \) denote the complement of \( \Omega_1 \cup \cdots \cup \Omega_{d-1} \) in \( X_{\alpha,d} \). 

\begin{Lem}
\label{lem:factorization}
Let \( \alpha \) be an ordinal, and let \( d \in \bn \). 
If \( f \in  \ker\chi_{\alpha,d} \), then there exist homeomorphisms \( f_0, f_1, \ldots, f_d \in \overline\F_{\alpha,d} \) satisfying 
\begin{itemize}
\item \( f_0 \in \F_{\alpha,d} \),
\item \( f_k \) is supported in \( \Omega_k \) for \( k \in \{1, \ldots, d\} \), and
\item \( f = f_0 \circ f_1 \circ \cdots \circ f_d \). 
\end{itemize} 
In particular, \( \ker \chi_{\alpha,d} \) is contained in \( \overline \F_{\alpha,d} \).
\end{Lem}

\begin{proof}
Fix \( f \in \ker \chi \). 
As \( |O_k(f) | = | I_k(f) | \), we can choose \( g_k \in F \) such that \( g_k(f(O_k(f))) = f(I_k(f)) \),  \( g_k(f(I_k(f))) = f(O_k(f)) \), and \( g_k \) fixes every other element of \( N \). 
Let \( g = g_1 \circ \cdots \circ g_{d-1} \), so that \( (g\circ f) (\Omega_k \cap N) = \Omega_k \cap N \) for all \( k \). 
 
Choose \( f_k \in \PH \) supported in \( \Omega_k \) such that \( f_k \) induces the same permutation of \( \Omega_k \cap N \) as does \( g \circ f \). 
It follows that \[ h = g \circ f \circ f_d^{-1} \circ \cdots \circ f_{1}^{-1} \]
induces a trivial permutation on \( N \); in particular, \( h \in \F \). 
Setting \( f_0 = g^{-1} \circ h \) yields the factorization \( f = f_0 \circ f_1 \circ \cdots f_d \) with \( f_0 \in \F \) and \( f_k \) supported in \( \Omega_k \) for \( k > 0 \).
It is left to verify that each of the \( f_k \) is in \( \overline F \). 

Let \( G_k \) denote the subgroup of \(  \PH \) consisting of elements supported in \( \Omega_k \). 
Note that \( f_k \in G_k \). 
We can identify \( G_k \) with \( \Homeo(\Omega_k) \), and as \( \Omega_k \) is homeomorphic to \( \omega^{\alpha+1} + 1 \) by \Cref{lem:stable}, we can apply \Cref{cor:max subgroup} to see that \( G_k \cap \overline F = G_k \), allowing us to conclude that \( f_k \in \overline F \), as desired. 
\end{proof}

\begin{Thm}
\label{thm:ses}
	Let \( \alpha \) be an ordinal, and let \( d \in \bn \ssm\{1\} \). 
	
	\begin{enumerate}[(1)]
	
		\item There exists a homomorphism \( \chi \co \PH_{\alpha,d} \to \bz^{d-1} \) such that 
				\[
					1 \longrightarrow \overline \F_{\alpha,d} \lhook\joinrel\longrightarrow \PH_{\alpha,d} \overset{\chi}\longrightarrow \bz^{d-1} \longrightarrow 1
				\]
				is a split short exact sequence. 
				
			\item The short exact sequence 
				\[
					1 \longrightarrow  \PH_{\alpha,d} \lhook\joinrel\longrightarrow \HH_{\alpha,d} \longrightarrow \Sym(d) \longrightarrow 1
				\]
				is split.
				
	\end{enumerate}
	
	Moreover, all the homomorphisms---including the sections---in the above exact sequences are continuous when \( \HH_{\alpha,d} \), \( \overline \F_{\alpha,d} \), and \( \PH_{\alpha,d} \) are equipped with the compact-open topology and \( \mathbb Z^{d-1} \) and \( \Sym(d) \) are equipped with the discrete topology.
\end{Thm}

\begin{proof}
	For (1), we have already established the section to \( \chi \) in \Cref{lem:section}; in \Cref{lem:in kernel}, we showed that \( \overline \F \) is contained in \( \ker \chi \); and in \Cref{lem:factorization}, we showed that \( \ker \chi \) is contained in \( \overline F \). 
	We have therefore established (1). 
	
	We move to the second sequence, which is exact by definition; we must show that it is split.
	Let \( J_1 = \{ x \in X_{\alpha,d}: x \leq \mu_1 \} \) and, for \( k \in \{2, \ldots, d\} \), let \( J_k = \{ x \in X_{\alpha,d} : \mu_{k-1} < x \leq \mu_k \} \). 
	Given \( j,k \in \{ 1, \ldots, d\} \), there is a unique order-preserving bijection \( g_{j,k} \co J_j \to J_k \). 
	Using our enumeration of \( M \), identify \( \Sym(M) \) with \( \Sym(d) \), and let \( \Phi\co \HH \to \Sym(d) \) be the homomorphism corresponding to the action of \( \HH \) on \( M \), so \( \ker\Phi = \PH \). 
	Taking \( \sigma \in \Sym(M) \), define \( \Psi(\sigma) \in \HH \) to be the homeomorphism  given by
	\[
		\Psi(\sigma)(x) =  g_{k,\sigma(k)}(x),
	\]
	whenever \( x \in J_k \). 
	Then \( \Psi \) is a section of \( \Phi \), establishing (2).

	Now to finish, we must establish the continuity of the homomorphisms in (1) and (2). 
	The inclusions are continuous by definition;
	therefore, it is only left to verify the continuity of \( \chi \) and \( \Phi \).
	Recall that (i) a homomorphism from a topological group to a discrete group is continuous if and only if the kernel is open, and (ii) a subgroup of a topological group is open if and only if it contains a neighborhood of the identity.
	By (1), the kernel of \( \chi \) is \( \overline \F \), which is readily verified to be open.
	Indeed, \( \overline \F \) contains the open neighborhood of the identity containing the elements that stabilize each of the \( \Omega_k \) setwise. 
	And, as \( \PH \) contains \( \overline \F \), which is open, \( \PH \) is open itself, implying \( \Phi \) is continuous.
\end{proof}

Before continuing, let us come back to noting that the definition of \( \chi \) did not depend on the choice of the \( \Omega_k \).
Indeed, let \( \Omega_1', \ldots, \Omega_k' \) be another such partition and let \( \chi' \) be the resulting homomorphism. 
Then there exists \( h \in \PH \) such that \( h(\Omega_k) = \Omega_k' \) for all \( k \). 
It follows that \( \chi' = \chi \circ i_h \), where \( i_h \)  denotes conjugation by \( h \). 
As the codomain is abelian, we deduce that \( \chi' = \chi \). 

\subsection{Properties of \( \overline \F_{0,d} \)}
\label{sec:F0}

Given \Cref{thm:ses}, it is clear that we need to understand the structure of \( \overline \F \) in order to understand the structure of \( \PH \) and \( \HH \).
Unfortunately, the subgroup \( \F \) gets in the way of our strategies.
We have the tools to investigate the quotient \( \overline \F_{\alpha,d} / \F_{\alpha,d} \), but unfortunately this quotient is independent of \( \alpha \); in particular, this quotient group is isomorphic to \( \overline \F_{0,d} \). 
Because of this, we now limit our focus to \( \overline \F_{0,d} \).

The key idea is to use fragmentation to reduce statements about \( \overline \F_{0,d} \) to statements about \( \Homeo(\omega) \). 
Using \Cref{lem:in kernel} together with \Cref{lem:factorization} and the fact that the choices of the \( \Omega_k \) above were arbitrary, we obtain the following version of \Cref{lem:factorization} that is well suited to our purposes.

\begin{Lem}[Fragmentation in \( \overline \F_{\alpha,d} \)]
\label{lem:factorization2}
Let \( \alpha \) be an ordinal, and let \( d \in \bn \). 
If \( f \in \overline \F_{\alpha,d} \), then for any pairwise-disjoint clopen neighborhoods \( U_1, \ldots, U_d \) of \( \mu_1, \ldots, \mu_d \), respectively, there exist homeomorphisms \( f_0, f_1, \ldots, f_d \) such that 
\begin{itemize}
\item \( f_0 \in \F_{\alpha,d} \),
\item \( f_i \) is supported in \( U_i \) for \( i \in \{1, \ldots, d\} \), and
\item \( f = f_0 \circ f_1 \circ \cdots \circ f_d \). 
\qed
\end{itemize}
\end{Lem}

We stated the fragmentation lemma for general \( \alpha \), but it is particularly useful for \( \alpha =0 \), as in this case \( \F_{0,d} \) is the subgroup of finitely supported homeomorphisms. 

\begin{Lem}
\label{lem:local isom}
Let \( \alpha \) be an ordinal, and let \( d \in \bn \).
If \( U \) is a clopen neighborhood of \( \mu \in M \) such that \( U \cap M = \{\mu\} \), then the subgroup of \( \overline \F_{\alpha,d} \) consisting of homeomorphisms supported in \( U \) is isomorphic to \( \Homeo(\omega^{\alpha+1}) \). 
\end{Lem}

\begin{proof}
Let \( G \) denote the subgroup of \( \overline \F_{\alpha,d} \) consisting of elements supported in \( U \). 
We can then identify \( G \) with a subgroup of \( \Homeo(U) \). 
By \Cref{lem:stable}, \( U \) is homeomorphic to \( X_{\alpha,1} \), and hence as  \( G \) contains an element inducing an infinite permutation on the set \( N_{\alpha,d} \), Theorem~\ref{thm:normal generators 1} implies \( G \) is isomorphic to \( \Homeo(\omega^{\alpha+1}) \). 
\end{proof}

\Cref{lem:factorization2} and \Cref{lem:local isom} imply that \( \overline \F_{\alpha,d} / \F_{\alpha,d} \) can be realized as a quotient of \( \Homeo(\omega^{\alpha+1})^d \), and hence this quotient is strongly distorted and uniformly perfect. 
In the case of uniform perfectness, the difficulty is then understanding whether we can express the elements of \( \F_{\alpha,d} \) as products of commutators of elements in \( \overline \F_{\alpha,d} \).
Unfortunately, we cannot answer this question in general, but we can do so for \( \alpha = 0 \). 

\begin{Thm}
\label{thm:perfect F}
If \( d \in \bn\ssm \{1\} \), then \( \overline \F_{0,d} \) is uniformly perfect of commutator width at most four. 
\end{Thm}

\begin{proof}
Fix \( f \in \overline \F_{0,d} \). 
Let \( U_1, \ldots, U_d \) be pairwise-disjoint clopen neighborhoods of \( \mu_1, \ldots, \mu_d \), respectively.
By \Cref{lem:factorization2}, we can write \( f= f_0 \circ f_1 \circ \cdots \circ f_d \) with \( f_0 \in \F \) and \( f_i \) supported in \( U_i \) for \( i > 0 \). 
Combining \Cref{lem:local isom} and \Cref{thm:mainthm2}, we can write \( f_i = [g_{i,1}, h_{i,1}][g_{i,2},h_{i,2}][g_{i,3},h_{i,3}] \) with \( g_{i,j} \) and \( h_{i,j} \) supported in \( U_i \) for \( i > 1 \).
Observe that if group elements \( a \), \( b \), \( c \), and \( d \) are such that \( a \) and \( b \) both commute with each of \( c \) and \( d \), then \( [a,b][c,d] = [ac,bd] \). 
It follows that \( f_1 \circ \cdots \circ f_d \)  can be expressed a product of three commutators. 
To finish, observe that as the support of \( f_0 \) is finite, we can conjugate \( f_0 \) to have support in \( U _1 \), allowing us to apply \Cref{lem:commutator} to express \( f_0 \) as a commutator.
Therefore, \( f \) is a product of four commutators, as desired. 
\end{proof}

We now turn to classifying the normal generators of \( \overline \F_{0,d} \). 
Given \( Z \subset M_{0,d} \), let \( \F_Z \) denote the subgroup of \( \overline \F_{0,d} \) consisting of homeomorphisms that restrict to the identity in an open neighborhood of each point of \( Z \).
Given a subgroup \( \Gamma \) of \( \overline \F_{0,d} \), we let \( Z_\Gamma \) be the subset of \( M_{0,d} \) consisting of the \( \mu \in M_{0,d} \) such that for each \( g \in \Gamma \) there is an open neighborhood \( U \) of \( \mu \) that is fixed pointwise by \( g \), that is, \( Z_\Gamma \) is the largest subset of \( M_{0,d} \) satisfying \( \F_{Z_\Gamma} < \Gamma \). 
For \( f \in \overline \F_{0,d} \), we write \( Z_f \) to denote \( Z_{\langle f \rangle} \).

\begin{Thm}
\label{thm:normal generators 2}
Let \( d \in \bn \ssm\{1\} \), and let \( \Gamma \) be a normal subgroup of \( \overline \F_{0,d} \).
If \( \Gamma \) is not contained in \( \F_{0,d} \), then  \( \F_{Z_\Gamma} = \Gamma \).
Moreover, (1) there exists \( \gamma \in \Gamma \) such that \( Z_\gamma = Z_\Gamma \), and (2) if \( \gamma \in \Gamma \) is such that \( Z_\gamma = Z_\Gamma \), then \( \gamma \) uniformly normally generates \( \Gamma \). 
\end{Thm}

\begin{proof}
By assumption, \( \Gamma \) is a normal subgroup of \( \overline \F \) not contained in \( \F \), and therefore, \( Z_\Gamma \neq M \). 
As the elements of \( N = N_{0,d} \) are isolated, it follows that \( \Gamma \leq F_{Z_\Gamma} \), so we need only show that \( F_{Z_\Gamma} \leq \Gamma \). 

First, we claim that if \( \mu \) belongs to \( M \) but not \( Z_\Gamma \), then there exists a clopen neighborhood \( U_\mu \) of \( \mu \) and an element \( g_\mu \) in \( \Gamma \) supported in \( U_\mu \) such that \( U_\mu \cap M = \{\mu\} \) and such that \( g_\mu \) induces an infinite permutation on \( N \cap U_\mu \). 

Indeed, fix \( \mu \in M \ssm Z_\Gamma \), and choose \( g \in \Gamma \) such that \( g \) induces a nontrivial permutation on \( U \cap N \) for every open neighborhood \( U \) of \( \mu \). 
By continuity, we can choose pairwise-disjoint clopen neighborhoods \( U_1, \ldots, U_d \) of \( \mu_1, \ldots, \mu_d \), respectively, such that \( g(U_i) \cap U_j = \varnothing \) whenever \( j \neq k \). 
Applying \Cref{lem:factorization2} to \( g \) allows us to write \( g = g_0 \circ g_1 \circ \cdots \circ g_d \) with \( g_0 \in \F \) and  with \( g_i \in \overline \F \) supported in \( U_i \).

Let \( k \in \{1, \ldots, d\} \) be such that \( \mu = \mu_k \). 
By \Cref{lem:displacement}, there exists an infinite subset \( A \) of \( N \cap U_k \) such that \( g_k(A) \cap A = \varnothing \).
By shrinking \( A \) we may assume that \( A \) is disjoint from the support of \( g_0 \). 
Let \( \tau \) be the involution such that \( \tau|_A = g_k|_A \), \( \tau|_{g_k(A)} = g_k^{-1}|_{g_k(A)} \),  and such that \( \tau \) is the identity on the complement of \( A \cup g_k(A) \). 
It follows that \( [\tau, g_0g_k] \)  induces an infinite permutation on \( U_k \cap N \), and 
\begin{align*}
	[\tau, g]	&= \tau \circ g \circ \tau^{-1} \circ g^{-1}  \\
			&= g_0^\tau \circ g_k^\tau \circ g_k^{-1} \circ g_0^{-1} \\
			&= [\tau, g_0g_k] . 
\end{align*}
The first equality implies that \( [\tau,g] \in \Gamma \), and the second equality is deduced from the fact that \( \tau \) commutes with each \( g_i \) for \( i > 0 \) and \( i \neq k \). 
As \( \tau \) is supported in \( U_k \), it follows from the second equality that \( [\tau,g] \) restricts to the identity on the complement of \( g_0(U_k) \), i.e., \( [\tau,g] \) is supported in \( g_0(U_k) \). 
Setting \( g_\mu = [\tau, g] \) and \( U_\mu = g_0(U_k) \), establishes the claim.

Now, let \( G_\mu \) denote the subgroup of \( \overline \F \) consisting of elements supported in \( U_\mu \). 
\Cref{lem:local isom} implies that \( G_\mu \) is isomorphic to \( \Homeo(\omega^{\alpha+1}) \). 
In particular, \Cref{thm:normal generators 1} implies that \( g_\mu \) uniformly normally generates \( G_\mu \). 
Moreover, as \( g_\mu \) belongs to \(\Gamma \) and \( \Gamma \) is normal, we have that \( G_\mu \) is a subgroup of \( \Gamma \).

Next, fix \( f \in \F_{Z_\Gamma} \). 
Again applying \Cref{lem:factorization2}, we can write \( f = f_0 \circ \prod_{\mu \in M \ssm Z_\Gamma} f_\mu \) with \( f_0 \in \F \) and each \( f_\mu \in \overline \F \) supported in \( U_\mu \).
It follows that there exists \( w \in \overline \F_{0,d} \) such that \( f \circ w = f_0 \in \F \) and such that \( w \) can be expressed as a product of \( 12 |M \ssm Z_\Gamma| \) elements from the set \( \{ g_\mu, g_\mu^{-1} :  \mu \in M \ssm Z_\Gamma \} \) and their conjugates (the number twelve comes from the proof of \Cref{thm:normal generators 1}). 
In particular, \( w \) belongs to \( \Gamma \). 
It is left to check that \( f_0 \) is an element of \( \Gamma \). 
As \( f_0 \in \F_{0,d} \) has finite support, there is a conjugate of \( f_0 \) supported in \( U_{\mu_1} \). 
\Cref{prop:anderson} implies that \( f_0 \) can be written as a product of four conjugates of \( g_{\mu_1} \) and \( g_{\mu_1}^{-1} \).
Therefore, \( F_{Z_\Gamma} \) is a subgroup of \( \Gamma \), as desired.

To finish, let \( \gamma \in \Gamma \) be such that \( Z_\gamma = Z_\Gamma \).  
First note that such a \( \gamma \) exists: the product of the \( g_\mu \) constructed above is such a \( \gamma \). 
We can now run the entirety of the argument above with \( g = \gamma \).
Then, by construction, each of the \( g_\mu \) is a product of a conjugate of \( \gamma \) and \( \gamma^{-1} \); the result follows. 
\end{proof}

We can now readily deduce a classification of the normal generators of \( \overline \F_{0,d} \). 

\begin{Thm}[Normal generators for \( \overline \F_{0,d} \)]
\label{thm:normal generators F}
Let \( d \in \bn \). 
For \( f \in \overline\F_{0,d} \), the following are equivalent:
\begin{enumerate}[(i)]
	\item \( f \) normally generates \( \overline\F_{0,d} \).
	\item \( f \) uniformly normally generates \( \overline\F_{0,d} \).
	\item \( Z_f = \varnothing \). 
	\item \( f \) induces an infinite permutation of the set \( \{ \omega\cdot k + \ell : \ell \in \bn \} \) for each \( k \in d \). 
\end{enumerate}
\end{Thm}

\begin{proof}
The implications (ii) to (i), (i) to (iv), and (iv) to (iii) are immediate.  
The implication (iii) to (ii) follows from \Cref{thm:normal generators 2} by setting \( \Gamma = \overline \F_{0,d} \) and \( \gamma = f \).
\end{proof}

\subsection{Properties of \( \PH_{0,d} \) and \( \HH_{0,d} \)}

In light of the short exact sequences in \Cref{thm:ses}, we can build on the uniform perfectness of and the classification of normal generators for \( \overline \F_{0,d} \) to compute the abelianization of and construct minimal normal generating sets for each of \( \PH_{0,d} \) and \( \HH_{0,d} \).

\begin{Thm}[Abelianization]
\label{thm:abelianization}

	If \( d \in \bn\ssm\{1\} \), then:
    
    	\begin{enumerate}[(1)]
    	
    		\item The abelianization of \( \PH_{0,d} \) is isomorphic to \( \bz^{d-1} \).
    		
    		\item The abelianization of \( \HH_{0,d} \) is isomorphic to \( \bz/2\bz \times \bz/2\bz \).
    		
    	\end{enumerate}
	
\end{Thm}

\begin{proof}

By \Cref{thm:ses}, we know \( \PH_{0,d} \cong \overline \F_{0,d} \rtimes \bz^{d-1} \), and by \Cref{thm:perfect F} that \( \overline\F_{0,d} \) is  perfect, so that the abelianization of \( \PH_{0,d} \) is \( \bz^{d-1} \). 
For \( \HH_{0,d} \), again as \( \overline\F_{0,d} \) is perfect, the abelianization of \( \HH_{0,d} \) factors through the abelianization of \( \bz^{d-1} \rtimes \Sym(d) \), where we are using \Cref{thm:ses} for the semi-direct product decomposition. 
We will do the computation through a presentation.  
Let \( s_1, \ldots, s_{d-1} \) be the shift homeomorphisms from the proof of \Cref{lem:section}, and let \( e_i \) be the projection of \( s_i \) under the quotient map \( \HH_{0,d} \to \bz^{d-1} \rtimes \Sym(d) \).
It follows that \( \{e_1, \ldots, e_{d-1} \} \) is a  free generating set for \( \bz^{d-1} \).
If we let \( \tau_i \in \Sym(d) \) denote the transposition of \( i \) and \( i +1 \), then building off a standard presentation for the symmetric group,  a presentation for \( \bz^{d-1} \rtimes \Sym(d) \) is given by 
\[
	\left\{ e_1, \ldots, e_{d-1}, \tau_1, \ldots, \tau_{d-1} :
	\begin{array}{ll}
		\tau_i^2 = 1 & \\ 
		{[e_i,e_j]} = 1 &  \\
		{[\tau_i,\tau_j] = 1} & \text{if } |i-j| > 1 \\
		\tau_i\tau_j\tau_i = \tau_j\tau_i\tau_j & \text{if } |i-j| = 1 \\
		{[\tau_i, e_j] = 1} & \text{if } i \neq j, j+1 \\
		\tau_i e_i \tau_i^{-1} = e_{i+1}   & \text{if } i \neq d-1\\
		\tau_i e_{i+1} \tau_i^{-1} = e_i & \text{if } i \neq d-1\\
		\tau_{d-1} e_{d-1} \tau_{d-1} = -e_{d-1} &
	\end{array} \right\}
\]
To compute the abelianization, we add the commutator relation for each generator.
Letting \( \pi \) denote the canonical homomorphism to the abelianization, we see that \( \pi(\tau_i) = \pi(\tau_j) \), \( \pi(e_i) = \pi(e_j) \), and \( \pi(e_i) = -\pi(e_i) \). 
It follows that the abelianization is generated by two commuting order two elements and is therefore isomorphic to \( \bz/2\bz \times \bz/2\bz \). 
\end{proof}

\begin{Thm}[Minimal cardinality for normal generating set]
\label{thm:cardinality}
Let \( \alpha \) be an ordinal, and let \( d \in \bn \). 
The minimal cardinality for a normal generating set of \( \PH_{0,d} \) is \( d-1 \) and is two for \( \HH_{0,d} \). 
\end{Thm}

\begin{proof}
Let us first consider \( \PH_{0,d} \). 
First note that, by \Cref{thm:abelianization}, the cardinality of any normal generating set for \( \PH_{0,d} \) must be at least \( d-1 \), as it projects to a generating set for \( \bz^{d-1} \).
So, we need only find a normal generating set of cardinality of \( d-1 \); to this end, let \( S = \{s_1, \ldots, s_{d-1} \} \), where the \( s_i \) are the shift homeomorphisms constructed in the proof of \Cref{lem:section}. 
We can then label the elements of \( N \) as \( \{ (n,i) : n \in \bn, 1 \leq i \leq d-1 \} \) in such a way that \( s_i(n,i) = (n+1,i) \) and \( s_i(n,j) = (n,j) \) whenever \( j \neq i \). 
Let \( s = s_1\circ s_2 \circ \cdots \circ s_{d-1} \). 
Choose \( t \in \PH_{0,d} \) such that \( t(2n,i) = (2n+1,i) \) and \( t(2n+1,i) = (2n,i) \). 
Then \( s^t(2n,i) = (2n+3,i) \) and \( s^t(2n+1,i) = (2n, i) \). 
In particular,
\[
	[t,s](2n+1,i) = (2n+3,i) \text{ and } [t,s](2n,i) = (2n-2,i) 
\]
It follows that \( [t,s] \in \overline \F_{0,d} \) and \( Z_{[t,s]} = \varnothing \); therefore, \( [t,s] \) normally generates \( \overline\F_{0,d} \) by \Cref{thm:normal generators F}. 
By \Cref{thm:ses}, \( \PH_{0,d} \) is the semi-direct product of \( \overline \F_{0,d} \) and the subgroup generated by \( S \), and hence \( S  \) is a normal generating set for \( \PH_{0,d} \). 

Turning to \( \HH_{0,d} \), \Cref{thm:abelianization} tells us that a normal generating set for \( \HH_{0,d} \) must have cardinality at least two.
The group \( \Sym(d) \) is normally generated by a single element (e.g., if \( d\neq 4 \), then every odd permutation is a normal generator). 
Let \( h \in \HH_{0,d} \) such that the induced permutation of \( h \) on \( \Sym(M) \) normally generates \( \Sym(M) \). 
Let \( s_1, \ldots, s_{d-1} \) and \( s \) be as above. 
Observe that \( s_i \) and \( s_j \) are conjugate in \( \HH_{0,d} \); in particular, \( s \) is a product of conjugates of \( s_1 \). 
Therefore, from above, we see that \( \PH_{0,d} \) is in the group normally generated by \( s_1 \) in \( \HH_{0,d} \). 
By \Cref{thm:ses}, \( \HH_{0,d} \) is isomorphic to \( \PH_{0,d} \rtimes \Sym(M) \), implying \( \{s_1, h\} \) is a normal generating set for \( \HH_{0,d} \). 
\end{proof}

\bibliographystyle{amsplain}
\bibliography{ordinals-bib}

\end{document}